\documentclass[11pt]{amsart}

\usepackage{amsmath, amsfonts, amssymb, amsthm, mathtools}
\usepackage{stmaryrd}
\usepackage{enumitem}
\usepackage{hyperref}
\usepackage{tikz-cd}
\usepackage[margin = 1in]{geometry}

\usepackage{hyperref}
\hypersetup{
	colorlinks = true,
	citecolor = blue,
  	linkcolor = blue,
  	urlcolor  = blue
}

\newtheorem{theorem}{Theorem}[section]
\newtheorem{corollary}[theorem]{Corollary}
\newtheorem{lemma}[theorem]{Lemma}
\newtheorem{proposition}[theorem]{Proposition}

\numberwithin{equation}{section}

\theoremstyle{definition}
\newtheorem{definition}[theorem]{Definition}
\newtheorem{notation}[theorem]{Notation}
\newtheorem{remark}[theorem]{Remark}
\newtheorem{example}[theorem]{Example}

\newtheorem*{ack}{Acknowledgments}

\newcommand{\ds}{\displaystyle}
\newcommand{\mc}{\mathcal}
\newcommand{\mf}{\mathfrak}

\newcommand{\ZZ}{\mathbb{Z}}
\newcommand{\CC}{\mathbb{C}}
\renewcommand{\phi}{\varphi}
\newcommand{\cok}{\operatorname{cok}}
\renewcommand{\hom}{\textup{Hom}}
\newcommand{\depth}{\operatorname{depth}}

\newcommand{\lcm}{\operatorname{lcm}}

\newcommand{\ord}{\operatorname{ord}}
\newcommand{\rank}{\operatorname{rank}}
\newcommand{\rad}{\operatorname{rad}}
\newcommand{\End}{\operatorname{End}}
\newcommand{\image}{\operatorname{Im}}
\newcommand{\ulcomp}{\operatorname{uc}}

\newcommand{\matfac}[3]{\operatorname{MF}_{#1}^{#2}(#3)}

\newcommand{\wotimes}[1]{\otimes^{#1}}

\title{Tensor products of $d$-fold matrix factorizations}

\author[Sheng]{Richie Sheng}
\address{Department of Mathematics, University of Utah, Salt Lake City, UT, USA}
\email{u1415944@utah.edu}

\author[Tribone]{Tim Tribone}
\address{Department of Mathematics, University of Utah, Salt Lake City, UT, USA}
\email{tim.tribone@utah.edu}
\urladdr{\href{https://timtribone.com/}{timtribone.com}}

\subjclass{13C14, 13C13, 13H10, 13F25}
\keywords{matrix factorizations, tensor product, Ulrich modules, maximal Cohen-Macaulay modules}

\begin{document}
    
    \begin{abstract}
        Consider a pair of elements $f$ and $g$ in a commutative ring $Q$. Given a matrix factorization of $f$ and another of $g$, the tensor product of matrix factorizations, which was first introduced by Kn\"orrer and later generalized by Yoshino, produces a matrix factorization of the sum $f+g$. We will study the tensor product of $d$-fold matrix factorizations, with a particular emphasis on understanding when the construction has a non-trivial direct sum decomposition. As an application of our results, we construct indecomposable maximal Cohen-Macaulay and Ulrich modules over hypersurface domains of a certain form.
    \end{abstract}

\maketitle

\section{Introduction}\label{sec:intro}

    Let $(Q,\mf n)$ be a regular local ring, $f$ a non-zero element of $\mf n^2$, and $z$ an indeterminate. In \cite{knorrer1987cohen}, Kn\"orrer introduced a construction which allowed for the transfer of \textit{maximal Cohen-Macaulay} (MCM) modules between the hypersurface rings $Q/(f)$ and $Q\llbracket z \rrbracket/(f+z^2)$. The construction, which is most easily stated in terms of matrix factorizations and is now commonly referred to as the tensor product of matrix factorizations, has had far-reaching applications. In particular, the careful analysis of the functor described by Kn\"orrer is at the heart of their famous periodicity theorem. 

    Kn\"orrer's functor is defined as follows: given a matrix factorization $(\phi,\psi)$ of $f$, that is, a pair of $n\times n$ square matrices with entries in $Q$ satisfying 
        \[\phi\psi = f\cdot I_n = \psi\phi,\]
    Kn\"orrer defines a new matrix factorization of $f+z^2$ given by the pair of $2n \times 2n$ block matrices
        \[(\phi,\psi) \otimes (z,z) \coloneqq \left(\begin{pmatrix}
            zI_n & \phi \\
            \psi & -zI_n
        \end{pmatrix}, \begin{pmatrix}
            zI_n & \phi\\
            \psi & -zI_n
        \end{pmatrix}\right).\]
        Yoshino later gave a generalization in \cite{yoshino1998tensor} which produces a matrix factorization of any sum $f+g$ starting from a factorization of $f$ and a factorization of $g$ using an analogous formula. In Yoshino's language, the above matrix factorization can be viewed as the tensor product of the matrix factorization $(\phi,\psi)$ of $f$ and of the $(1\times 1)$ matrix factorization $(z,z)$ of $z^2$. Yoshino's generalization has found applications in homological mirror symmetry \cite{seidel2010homological}, knot theory \cite{khovanov2008matrix}, and generally in the theory of MCM modules over hypersurface rings and complete intersection rings.

    Our goal is to study a further generalization of Yoshino's and, by extension, Kn\"orrer's construction. Namely, we study the tensor product of $d$-fold matrix factorizations. The techniques and ideas introduced in \cite{yoshino1998tensor} serve as a framework for this article. We lift Yoshino's formulas, and some of their proofs, to the case of $d$-fold matrix factorizations for arbitrary $d \ge 2$. The overall theme remains the same: given a matrix factorization of an element $f$ and another of $g$, the tensor product produces a matrix factorization (with the same number of factors) of the sum $f+g$. 
    
    It is worth noting that a $d$-fold version of Yoshino’s tensor product has appeared in various other forms (several of which actually predate Yoshino's paper); see \cite{childs1978linearizing}, \cite{backelin1986matrix}, \cite{backelin1989ulrich},  \cite{herzog1991linear}, \cite{blaser2017ulrich}, and \cite{hopkins2021n}. The version we describe most closely resembles the tensor product of complexes (really of ``$d$-complexes'') and therefore we feel it serves as a good candidate for a unified definition. We study some of the basic properties of the tensor product, focusing mostly on those that are relevant to our main goal: understanding its decomposability. As it turns out, the tensor product of two indecomposable matrix factorizations often has many non-trivial direct summands, so, understanding when the operation preserves indecomposability is the main theme of Section \ref{sec:disjoint_vars}. We also give one predictable situation in which the tensor product always decomposes in Section \ref{sec:knorrer_decomp}.
    
    Our main application is to the theory of \textit{Ulrich} modules over hypersurface domains. The authors of \cite{herzog1991linear} showed that any hypersurface ring (and more generally, any \textit{strict complete intersection}) has an Ulrich module using a construction essentially equivalent to the $d$-fold tensor product. In Section \ref{sec:Ulrich}, we reproduce their result using elementary properties of the tensor product in the case of a hypersurface domain, and because of the results of Section \ref{sec:disjoint_vars}, we are able to give conditions (and examples) for when the Ulrich modules produced are indecomposable.

\section{Background and notation}\label{sec:prelim}
    
    Let $Q$ be a commutative ring (always assumed to be Noetherian), $d \ge 2$ an integer, and fix an element $f \in Q$. In this section, we provide a brief background on $d$-fold matrix factorizations and set up the notations that will be needed throughout. We also record several facts about the structure of the category of $d$-fold matrix factorizations which will be needed in Section \ref{sec:disjoint_vars}.

    \begin{definition}
        A \textit{matrix factorization} of $f$ (\textit{with $d$ factors}) is a sequence of $Q$-homomorphisms between finitely generated projective $Q$-modules
            \[\begin{tikzcd}
               X = ( X_0 \rar{\phi_0} &X_{d-1} \rar{\phi_{d-1}} &\cdots \rar{\phi_2} &X_1 \rar{\phi_1} &X_0 )
            \end{tikzcd}\] satisfying
            \begin{equation}\label{eq:dMF}
                \phi_i\phi_{i+1}\cdots\phi_0\phi_1\cdots\phi_{i-2}\phi_{i-1} = f\cdot 1_{X_{i-1}}
            \end{equation} for all $i \in \ZZ_d \coloneqq \ZZ/d\ZZ$ (see \eqref{sec:roots_of_1}). When convenient, we will suppress the full sequence and refer to a matrix factorization as an ordered $d$-tuple, $X = (\phi_1,\phi_2,\dots,\phi_{d-1},\phi_0)$. Alternatively, we will view a matrix factorization $X = (X,\phi)$ as a $\ZZ_d$-graded projective $Q$-module, $X = \bigoplus_{k \in \ZZ_d}X_k$, equipped with a degree $-1$ map $\phi: X \to X$ satisfying $\phi^d=f\cdot 1_X$. 
            
            The collection of all $d$-fold matrix factorizations of $f$ forms an additive category which we will denote by $\matfac{Q}{d}{f}$. Namely, if $(X,\phi), (X',\phi') \in \matfac{Q}{d}{f}$, then the direct sum is given by $(X,\phi)\oplus (X',\phi') \coloneqq (X\oplus X', \phi\oplus \phi')$. We will call a non-zero object in $\matfac{Q}{d}{f}$ \textit{indecomposable} if it cannot be decomposed into a direct sum of non-zero matrix factorizations. A morphism $\alpha \in \hom(X,X')$ in $\matfac{Q}{d}{f}$ is a $d$-tuple of $Q$-homomorphisms, $\alpha = (\alpha_0,\alpha_1,\dots,\alpha_{d-1})$, such that $\alpha_{k-1}\phi_k = \phi_k'\alpha_k$ for all $k \in \ZZ_d$. Equivalently, $\alpha: X \to X'$ is a degree zero $Q$-linear map satisfying $\alpha\phi = \phi'\alpha$. More details on the category $\matfac{Q}{d}{f}$ can be found in \cite{hopkins2021n} and \cite{tribone2022matrix}.
    \end{definition}

    If $f$ is a nonzerodivisor and each $X_k$, $k \in \ZZ_d$, is a free $Q$-module, then \eqref{eq:dMF} forces each of the ranks of the free $Q$-modules $X_k$ to be equal \cite[Corollary 5.4]{eisenbud1980homological}. To see this, notice that, for each $k \in \ZZ_d$, the pair $(\phi_k,\phi_{k+1}\phi_{k+2}\cdots \phi_0\phi_1 \cdots \phi_{k-1})$ forms a matrix factorization of $f$ with two factors. Hence Eisenbud's result implies that $\rank_Q X_k = \rank_Q X_{k-1}$ for all $k \in \ZZ_d$. In this case, we say that the matrix factorization $X$ has \textit{rank} $n$ where $n = \rank_Q X_k$ for all $k \in \ZZ_d$.

    The \textit{shift} of a matrix factorization $(X,\phi) \in \matfac{Q}{d}{f}$ will refer to the factorization obtained by cyclically permuting the maps $\phi_k$. We will denote this factorization by
        \[TX = T(\phi_1,\phi_2,\dots,\phi_{d-1},\phi_0) \coloneqq (\phi_2,\phi_3,\dots,\phi_0,\phi_1).\] 
    In particular, for $k \in \ZZ_d$, the degree $k$ piece of $TX$ is $(TX)_k = X_{k+1}$. Similarly, the $i$th shift of $X$ cyclically permutes $X$ $i$ times:
        \[T^iX = (\phi_{i+1},\phi_{i+2},\dots,\phi_0,\phi_1,\dots,\phi_{i-1},\phi_i).\]
    
    We will refrain from using typical suspension notation as the shift operation $T(-)$ does not induce the suspension functor in the triangulated stable category associated to $\matfac{Q}{d}{f}$ unless $d=2$ (see \cite[Section 2]{tribone2022matrix}). 

    \subsection{Notation and conventions}\label{sec:roots_of_1} The main construction we will study is a $\ZZ_d$-graded analogue of the tensor product of complexes (technically of \textit{$d$-complexes}; see \cite[Section 3]{heller2021free}). In this setting, the usual sign convention is replaced by appropriate roots of unity. In order to give the construction in our desired generality, we will assume throughout that the ring $Q$ is a $\ZZ[\zeta]$-algebra where $\zeta$ is a primitive $d^{\textit{th}}$ root of unity in $\CC$. That is, we will assume that there is a (unital) ring homomorphism $\ZZ[\zeta] \to Q$. For instance, if $Q$ is a domain containing a primitive $d^{\textit{th}}$ root of unity (in the usual sense), then such an algebra structure exists. We will abuse notation and write $\zeta \in Q$ to denote the image of $\zeta$ under this homomorphism. 
    
    Due to the cyclic nature of $d$-fold matrix factorization, it will be convenient to take indices modulo $d$. Throughout, the reader should assume that subscripts and other indexing variables are to be interpreted modulo $d$ unless otherwise indicated. In particular, components of a morphism will (almost always) be taken modulo $d$ which we explain below.
        
    If $\alpha: \bigoplus_{j=1}^tM_j \to \bigoplus_{i=1}^sN_i$ is a morphism in an additive category, we use the notation $\alpha(i,j)$ to denote the composition
        \[M_j \hookrightarrow \bigoplus_{j=1}^tM_j \xrightarrow{\, \alpha \,} \bigoplus_{i=1}^sN_i \twoheadrightarrow N_i.\]
    We refer to $\alpha(i,j)$ as the $(i,j)$-component of $\alpha$ with respect to the given direct sum decompositions. In this case, $\alpha$ can be represented by a matrix $\alpha = (\alpha(i,j))_{1\leq i \leq s, 1\leq j \leq t}$. Furthermore, if both of the direct sum decompositions are indexed over $\ZZ_d$ with respect to the ordering $1 \prec 2 \prec \cdots \prec d-1 \prec d=0$, then the components of $\alpha$ are naturally indexed modulo $d$ as well, that is, for any $1\leq i,j\leq d$ and $s,t \in \ZZ$,
        \(\alpha(i+sd,j+td) = \alpha(i,j).\)
    The choice of ordering is simply to index the rows and columns of matrices in the usual way. In other words, as a matrix, $\alpha$ is written as 
        \[(\alpha(i, j))_{i, j \in \ZZ_d} = \begin{pmatrix}
            \alpha(1, 1) &  \alpha(1, 2) & \dots & \alpha(1, d) \\
            \alpha(2, 1) & \alpha(2, 2) & \dots & \alpha(2, d) \\
                  \vdots & \vdots & \ddots & \vdots \\
            \alpha(d, 1) & \alpha(d, 2) & \dots & \alpha(d,d)
        \end{pmatrix}.\]

\subsection{Idempotents in the category of matrix factorizations} 

    Before proceeding, we record two important facts about morphisms of matrix factorizations that will be needed in Section \ref{sec:disjoint_vars}. They both hold in greater generality than the setting of Section \ref{sec:disjoint_vars} so we include them here.
    
    Let $Q$ be a commutative ring, $d \ge 2$, and $f \in Q$. Suppose $(X,\phi) \in \matfac{Q}{d}{f}$ is of rank $n$. An idempotent $e \in \End(X)$ defines a matrix factorization 
        \(e(X) \coloneqq (\image e, \phi\vert_{\image E}) \in \matfac{Q}{d}{f}.\)
    A key aspect to the proofs in Section \ref{sec:disjoint_vars} will be to track the rank of $e(X)$. Here we will show that, when $Q$ is local, idempotents in $\matfac{Q}{d}{f}$ can be put into a normal form which identifies the rank of the matrix factorization $e(X)$.
    
    We will say that a pair of morphisms $\alpha \in \hom(X,X')$ and $\beta \in \hom(Y,Y')$ are \textit{equivalent} if there exists isomorphisms $\gamma$ and $\delta$ such that $\delta\alpha = \beta\gamma$. In this case, we will write $\alpha \sim \beta$. Notice that equivalent idempotents induce isomorphic summands, that is, if $e \sim e'$, then $e(X) \cong e'(X')$. 

    \begin{lemma}\label{thm:equiv_morphisms}
        Let $(X,\phi),(X',\phi')\in \matfac{Q}{d}{f}$ and suppose $\alpha=(\alpha_0,\dots,\alpha_{d-1}) \in \hom(X,X')$. For any $k \in \ZZ_d$, replacing $\alpha_k$ with $A\alpha_k B$, for isomorphisms $A$ and $B$, results in a morphism equivalent to $\alpha$.
    \end{lemma}

    \begin{proof}
        Let $k \in \ZZ_d$. For any isomorphisms of $Q$-modules $A:X_k' \to X_k'$ and $B:X_k \to X_k$, we have a commutative diagram
            \[\begin{tikzcd}[column sep=4em]
                X_{k+1} \rar{B^{-1}\phi_{k+1}} \dar{\alpha_{k+1}}  &X_k \rar{\phi_{k}B} \dar{A\alpha_k B} &X_{k-1} \dar{\alpha_{k-1}}\\
                X_{k+1}' \rar{A\phi_{k+1}'} &X_k' \rar{\phi_{k}'A^{-1}} &X_{k-1}'.
            \end{tikzcd}\]
        Let $Y=(\phi_1,\dots,\phi_{k}B,B^{-1}\phi_{k+1},\dots,\phi_0)$ and $Y' = (\phi_1',\dots,\phi_{k}'A^{-1},A\phi_{k+1}',\dots,\phi_0)$. Then $Y,Y' \in \matfac{Q}{d}{f}$ and the commutative diagram above shows that 
            \[\beta = (\alpha_0,\alpha_1,\dots,A\alpha_kB,\dots,\alpha_{d-1}):Y \to Y'\] 
        is a morphism of matrix factorizations. Now, if we let $\gamma = (1_{X_0},\dots,B^{-1},\dots,1_{X_{d-1}})$ and $\delta = (1_{X_0'},\dots,A,\dots,1_{X_{d-1}'})$ we have that $\delta\alpha = \beta\gamma$ implying that $\alpha$ and $\beta$ are equivalent.
    \end{proof}

    The next lemma provides a normal form for idempotents in $\matfac{Q}{d}{f}$. In particular, the lemma shows that idempotents split in the additive category $\matfac{Q}{d}{f}$.
    
    \begin{lemma}\label{thm:idempotent_normal_form}
       Assume that $(Q,\mf n)$ is local and $f \in \mf n$ is a nonzerodivisor in $Q$. If $X\in \matfac{Q}{d}{f}$ is of rank $n$ and $e=(e_0,\dots,e_{d-1})\in \End(X)$ is an idempotent, then there exists an integer $0\leq r \leq n$ such that $e$ is equivalent to a morphism $(\epsilon,\epsilon,\dots,\epsilon)$ where
            \[\epsilon = \begin{pmatrix}
                I_r &0\\
                0 &0
                \end{pmatrix}: Q^r\oplus Q^{n-r} \to Q^r\oplus Q^{n-r}.\]
        In particular, $e(X)$ is of rank $r$, $(1-e)(X)$ is of rank $n-r$, and $X \cong e(X) \oplus (1-e)(X)$.
    \end{lemma}

    \begin{proof}
        If $e=0$ or $e=1$, there is nothing to prove so assume $e\neq 0,1$. Since $e$ is an idempotent of matrix factorizations, each component satisfies $e_k^2 = e_k$, that is, the $Q$-homomorphism $e_k:X_k \to X_k$ is an idempotent for each $k \in \ZZ_d$. Therefore, there exists isomorphisms $A_k,B_k$ such that
            \[A_ke_k B_k = \begin{pmatrix}
                I_{r_k}&0\\
                0 &0
            \end{pmatrix}:Q^{r_k}\oplus Q^{n-r_k} \to Q^{r_k}\oplus Q^{n-r_k}\] 
        for some $0 < r_k < n$. Applying Lemma \ref{thm:equiv_morphisms} to each $e_k$, $k \in \ZZ_d$, we may assume that
                \[e = \left(\begin{pmatrix}
                    I_{r_0} &0\\
                    0 &0
                \end{pmatrix}, \begin{pmatrix}
                    I_{r_1} &0\\
                    0 &0
                \end{pmatrix},\dots ,\begin{pmatrix}
                    I_{r_{d-1}} &0\\
                    0 &0
                \end{pmatrix}\right)\] 
        for integers $0<r_0,r_1,\dots,r_{d-1} < n$, and that $\phi_k : Q^{r_{k}}\oplus Q^{n-r_{k}} \to Q^{r_{k-1}}\oplus Q^{n-r_{k-1}}$ for each $k \in \ZZ_d$. 

        To finish the proof, it suffices to show that $r_{i-1} = r_{i}$ for each $i \in \ZZ_d$. Let $i \in \ZZ_d$ and consider the commutative diagram
            \[\begin{tikzcd}[column sep = 5.5em]
                Q^{r_{i-1}}\oplus Q^{n-r_{i-1}} \rar{\phi_{i+1}\phi_{i+2}\cdots\phi_{i-1}} \dar{e_{i-1}} &Q^{r_{i}}\oplus Q^{n-r_{i}} \dar{e_{i}} \rar{\phi_i} &Q^{r_{i-1}}\oplus Q^{n-r_{i-1}} \dar{e_{i-1}}\\
                Q^{r_{i-1}}\oplus Q^{n-r_{i-1}} \rar{\phi_{i+1}\phi_{i+2}\cdots\phi_{i-1}} &Q^{r_{i}}\oplus Q^{n-r_{i}} \rar{\phi_i} &Q^{r_{i-1}}\oplus Q^{n-r_{i-1}}.
            \end{tikzcd}\] 
        Decompose $\phi_i$ and $\phi_{i+1}\phi_{i+2}\cdots\phi_{i-1}$ with respect to these direct sum decompositions into
            \[\phi_i = \begin{pmatrix}
                A &B\\
                C &D
            \end{pmatrix} \quad \text{ and } \quad \phi_{i+1}\phi_{i+2}\cdots\phi_{i-1} =\begin{pmatrix}
                A' &B'\\
                C' &D'
            \end{pmatrix}.\] 
        Since $(\phi_i,\phi_{i+1}\phi_{i+2}\cdots\phi_{i-1}) \in \matfac{Q}{2}{f}$, we have that \[\begin{pmatrix}
                        A &B\\
                        C &D
                    \end{pmatrix}
                    \begin{pmatrix}
                        A' &B'\\
                        C' &D'
                    \end{pmatrix} = \begin{pmatrix}
                        f\cdot I_{r_{i-1}} &0\\
                        0 & f\cdot I_{n-r_{i-1}}
                    \end{pmatrix}\] and \[\begin{pmatrix}
                        A' &B'\\
                        C' &D'
                    \end{pmatrix}\begin{pmatrix}
                        A &B\\
                        C &D
                    \end{pmatrix} = \begin{pmatrix}
                        f\cdot I_{r_{i}} &0\\
                        0 & f\cdot I_{n-r_{i}}
                    \end{pmatrix}.\]
        Using the fact that $e_{i-1}\phi_i = \phi_i e_{i}$, we can conclude that $B=0 $ and $C=0$. Similarly, the commutativity of the left square implies that $B'=0 $ and  $C' = 0$. Thus, $AA' = f\cdot I_{r_{i-1}}$ and $A'A = f \cdot I_{r_{i}}$ which is only possible if $r_{i-1} = r_{i}$ (see \cite[Corollary 5.4]{eisenbud1980homological}).

        The final statement follows by decomposing each $\phi_i, i \in \ZZ_d$, along the direct sum decomposition $\phi_i:Q^r \oplus Q^{n-r} \to Q^r \oplus Q^{n-r}$ where $r$ is the common value $r=r_0=r_1=\cdots=r_{d-1}$.
    \end{proof}

    We finish this section with a straightforward lemma that will also be useful in Section \ref{sec:disjoint_vars}. 

    \begin{lemma}\label{thm:scaled_by_units}
        Let $c_1,c_2,\dots,c_{d-1},c_0$ be elements of $Q$ such that $c_1c_2\dots c_{d-1}c_0 = 1$. Then, for any $(X,\phi) \in \matfac{Q}{d}{f}$, $X$ is isomorphic to the factorization 
            \[(c_1\phi_1,c_2\phi_2,\dots,c_{d-1}\phi_{d-1},c_0\phi_0) \in \matfac{Q}{d}{f}.\]
        In particular, if $c \in Q$ satisfies $c^d = 1$, then $cX \coloneqq (c\phi_1,\dots,c\phi_0) \cong X$. \qed
    \end{lemma}
    
\section{The tensor product}\label{sec:tensor_prod}

    In this section, we define the tensor product of $d$-fold matrix factorizations and prove some of the basic properties that will be needed going forward. Our definition closely follows the presentation given in \cite{hopkins2021n} with minor modifications. Throughout this section, let $Q$ be commutative ring, $d \ge 2$ an integer, and assume that $Q$ is a $\ZZ[\zeta]$-algebra where $\zeta \in \CC$ is a primitive $d^{\textit{th}}$ root of unity. We also fix elements $f$ and $g$ in $Q$.

    \begin{definition}
         For matrix factorizations $X = (X,\phi) \in \matfac{Q}{d}{f}$ and $Y = (Y,\psi) \in \matfac{Q}{d}{g}$ we define the \textit{tensor product} $X \wotimes{\zeta} Y = (X \wotimes{\zeta} Y,\Phi) \in \matfac{Q}{d}{f+g}$ as follows. The underlying $\ZZ_d$-graded projective $Q$-module is given by $X \wotimes{\zeta} Y = \bigoplus_{k \in \ZZ_d}(X \wotimes{\zeta} Y)_k$ where
            \[(X \wotimes{\zeta} Y)_k = \bigoplus_{i+j \equiv k \mod d} X_i \otimes_Q Y_j.\] 
        For homogeneous elements $x \in X$ and $y \in Y$, the map $\Phi$ is given by
            \[\Phi(x\otimes y) = \phi(x) \otimes y + \zeta^{|x|}x\otimes \psi(y)\]
        where $|x|$ denotes the degree of the element $x \in X$, i.e. $x \in X_i$ if and only if $|x| = i$. 
        
        In general, the isomorphism class of $X \wotimes{\zeta} Y$ depends on the choice of root of unity, $\zeta \in Q$ (for instance, see Example \ref{ex:non_commutatitve_otimes}). However, when the root of unity is fixed, we will omit $\zeta$ from the notation.
    \end{definition}
    
    This construction has appeared in various similar forms in \cite{childs1978linearizing}, \cite{backelin1986matrix}, \cite{backelin1989ulrich},  \cite{herzog1991linear}, \cite{blaser2017ulrich}, and \cite{hopkins2021n}. Each version has the same goal: to produce a matrix factorization of $f+g$ from a factorization of $f$ and a factorization of $g$. The techniques in each of the articles listed above are quite distinct, though some of the constructions can be identified with each other up to a choice of roots of unity and organization of the matrices (for example, see Proposition \ref{thm:otimes_alternate}). Our definition closely follows \cite[Section 2.5]{hopkins2021n} except for a slightly different approach to the needed roots of unity. The proof given by Hopkins explicitly uses the fact that $\zeta \in Q$ is a \textit{distinguished}\footnote{A \textit{distinguished} primitive $d^{\textit{th}}$ root of unity is an element $\zeta \in Q$ such that the $\sum_{i=0}^{n-1} \zeta^i$ is invertible in $Q$ for all $1 \leq n \leq d-1$ and zero for $n=d$. For example, $i \in \ZZ[i]$ is a primitive fourth root of unity but is not distinguished since $1+i$ is not invertible in $\ZZ[i]$.} primitive $d^{\textit{th}}$ root of unity so we are unable to transfer their proof to our setting. However, with slight modification, the proof of \cite[Proposition 2.2]{blaser2017ulrich} applies in our generality. 

    \begin{proposition}\cite{blaser2017ulrich,hopkins2021n}
        Let $X \in \matfac{Q}{d}{f}$ and $Y \in \matfac{Q}{d}{g}$. Then $X \wotimes{} Y$ defines a matrix factorization of $f+g$ in $\matfac{Q}{d}{f+g}$.
    \end{proposition}

        \begin{proof}
            Let $\rho : \ZZ[\zeta] \to Q$ denote the ring homomorphism defining the $\ZZ[\zeta]$-algebra structure on $Q$ and extend this homomorphism to $\rho: (\ZZ[\zeta])[x] \to Q[x]$ by setting $\rho(x) = x$. Notice that $x^d-1 = \rho(x^d-1) = \prod_{i=0}^{d-1}(x-\rho(\zeta)^i)$ in $Q[x]$. Hence the coefficient of $x^k$ on the right is zero for each $0 < k < d$. By picking $z_i = \rho(\zeta)^i$, $0 \leq i \leq d-1$, one can now follow the proof of \cite[Proposition 2.2]{blaser2017ulrich} with minor adjustments to conclude that $X \wotimes{} Y \in \matfac{Q}{d}{f+g}$.
        \end{proof} 

   The operation $( - ) \wotimes{} ( - ) : \matfac{Q}{d}{f} \times \matfac{Q}{d}{g} \to \matfac{Q}{d}{f+g}$ is functorial in both arguments. In particular, if $\alpha : X \to X'$ is a morphism in $\matfac{Q}{d}{f}$, then, for homogeneous $x\in X$ and $y \in Y$,
        \((\alpha \otimes 1_Y )(x\otimes y) = \alpha(x) \otimes y\)
    forms a morphism $\alpha \otimes 1_Y: X \wotimes{} Y \to X' \wotimes{} Y$ in $\matfac{Q}{d}{f+g}$. Similarly, if $\beta: Y \to Y'$ is a morphism in $\matfac{Q}{d}{g}$, then $1_X \wotimes{} \beta : X \otimes Y \to X \wotimes{} Y'$ defines a morphism in $\matfac{Q}{d}{f+g}$.
    
    It will often be useful to consider the components of the tensor product with respect to a given direct sum decomposition. For each $k \in \ZZ_d$, we will view $\Phi_k$ with respect to the direct sum decompositions $\Phi_k : \bigoplus_{j = 1}^{d} X_{j-1}\otimes Y_{k+1-j} \to \bigoplus_{i = 1}^{d} X_{i-1}\otimes Y_{k-i}$. For $1\leq i,j\leq d$, the $(i,j)$-component of $\Phi_k$ is given by
        \begin{equation}\label{eq:tensor}
            \Phi_k(i,j) = \begin{cases}
            \zeta^{i-1} 1_{X_{i-1}} \otimes \psi_{k+1-i} & \text{ if } j \equiv i \mod d\\
            \phi_i \otimes 1_{Y_{k-i}} & \text{ if } j \equiv i+1 \mod d\\
            0 & \text{ otherwise }
            \end{cases}.
        \end{equation}

    \begin{example}
        Let $Q = \Bbbk\llbracket x_i,y_i,z_i : i \in \ZZ_3 \rrbracket$ where $\Bbbk$ is a field containing a primitive third root of unity, $\zeta \in \Bbbk$. Consider the rank one matrix factorizations $X = (x_1,x_2,x_0) \in \matfac{Q}{3}{x_1x_2x_0},$ $Y = (y_1,y_2,y_0) \in \matfac{Q}{3}{y_1y_2y_0}$, and $Z = (z_1,z_2,z_0) \in \matfac{Q}{3}{z_1z_2z_0}$. Then $X \wotimes{\zeta} Y \in \matfac{Q}{3}{x_1x_2x_0 + y_1y_2y_0}$ is the triple of $3 \times 3$ matrices
            \[(A_1,A_2,A_0) = \left( \begin{pmatrix}
                y_1 & x_1 & 0 \\
                0 & \zeta y_0 & x_2\\
                x_0 & 0 & \zeta^2 y_2
            \end{pmatrix},\begin{pmatrix}
                y_2 & x_1 & 0 \\
                0 & \zeta y_1 & x_2\\
                x_0 & 0 & \zeta^2 y_0
            \end{pmatrix},\begin{pmatrix}
                y_0 & x_1 & 0 \\
                0 & \zeta y_2 & x_2\\
                x_0 & 0 & \zeta^2 y_1
            \end{pmatrix}\right) .\]
        If we tensor this with $Z$, we get a rank $9$ factorization of the sum $f = x_1x_2x_0 + y_1y_2y_0 + z_1z_2z_0$, that is, $(X\wotimes{\zeta} Y) \wotimes{\zeta} Z \in \matfac{Q}{3}{f}$ is given by triple of matrices
            \[(B_1,B_2,B_0) = \left( \begin{pmatrix}
                z_1I_3 & A_1 & 0 \\
                0 & \zeta z_0I_3 & A_2\\
                A_0 & 0 & \zeta^2 z_2I_3
            \end{pmatrix},\begin{pmatrix}
                z_2I_3 & A_1 & 0 \\
                0 & \zeta z_1I_3 & A_2\\
                A_0 & 0 & \zeta^2 z_0I_3
            \end{pmatrix},\begin{pmatrix}
                z_0I_3 & A_1 & 0 \\
                0 & \zeta z_2I_3 & A_2\\
                A_0 & 0 & \zeta^2 z_1I_3
            \end{pmatrix}\right)\]
        where $I_3$ is a $3\times 3$ identity matrix. In Sections \ref{sec:disjoint_vars} and \ref{sec:Ulrich} we show that both of these matrix factorizations are indecomposable and that each of the corresponding cokernel modules are \textit{indecomposable Ulrich modules} over the appropriate hypersurface ring.
    \end{example}

    \subsection{Properties of the tensor product}

    As above, we let $d \ge 2$ be an integer and $Q$ be a commutative $\ZZ[\zeta]$-algebra where $\zeta \in \CC$ is a primitive $d^{\textit{th}}$ root of unity. We also fix elements $f,g \in Q$. Our first observation is on the symmetry of the tensor product operation. Note that since $\ZZ[\zeta] = \ZZ[\zeta^{-1}]$, we can also take the tensor product with respect to $\zeta^{-1}$.

    \begin{proposition}\label{thm:otimes_alternate}
        Let $X \in \matfac{Q}{d}{f}$ and $Y \in \matfac{Q}{d}{g}$. Then there is a natural isomorphism of matrix factorizations $X\wotimes{\zeta} Y \cong Y \wotimes{\zeta^{-1}} X$ in $\matfac{Q}{d}{f+g}$.
    \end{proposition}

    \begin{proof}
        Define a map $\alpha: (X\wotimes{\zeta} Y,\Phi) \to (Y \wotimes{\zeta^{-1}} X, \Phi')$ given by
        \(\alpha(x \otimes y) = \zeta^{|x||y|}y\otimes x\) for homogeneous elements $x \in X$ and $y \in Y$. Then $\alpha$ is invertible and satisfies
            \begin{align*}
                \alpha\Phi(x \otimes y) 
                &= \zeta^{|\phi(x)||y|}y \otimes \phi(x) + \zeta^{|x|}(\zeta^{|x||\psi(y)|} \psi(y) \otimes x)\\
                &= \zeta^{(|x| - 1)|y|}y \otimes \phi(x) + \zeta^{|x|(1 + |\psi(y)|)}\psi(y) \otimes x\\
                &= \zeta^{-|y|}\left( \zeta^{|x||y|}y \otimes \phi(x)\right) + \zeta^{|x||y|}\psi(y) \otimes x\\
                &= \Phi'\alpha(x \otimes y)
            \end{align*}
        for any homogeneous $x \in X$ and $y \in Y$. This implies that $X \wotimes{\zeta} Y \cong Y\wotimes{\zeta^{-1}} X$ in $\matfac{Q}{d}{f+g}$. The fact that $\alpha$ induces a natural isomorphism (in both components) can be checked directly from the formula. 
    \end{proof}

    \begin{remark}
        In the case $d=2$, the conclusion of Proposition \ref{thm:otimes_alternate} amounts to the isomorphism $X \wotimes{} Y \cong Y \wotimes{} X$ \cite[Lemma 2.1]{yoshino1998tensor}. However, this is not the conclusion being made in the case that $d > 2$; in general, the tensor product construction with respect to a fixed root of unity is not symmetric when $d > 2$.
    \end{remark}

    \begin{example}\label{ex:non_commutatitve_otimes}
        Let $Q = \Bbbk[x,y,z,u,v,w]$ for a field $\Bbbk$ containing a primitive $3$rd root of unity, $\zeta \in \Bbbk$. Consider the rank one factorizations $X = (x,y,z) \in \matfac{Q}{3}{xyz}$ and $Y = (u,v,w) \in \matfac{Q}{3}{uvw}$. We claim that
        \(X\wotimes{\zeta} Y\) and \(Y \wotimes{\zeta} X\) are not isomorphic in $\matfac{Q}{3}{xyz+uvw}$.
    \end{example}

    \begin{proof}
        Assume that $(X \wotimes{\zeta} Y, \Phi) \cong (Y \wotimes{\zeta} X, \Phi')$, that is, assume that there exist invertible matrices $(\alpha_0, \alpha_1, \alpha_2)$ with entries in $Q$ such that $\alpha_{k-1}\Phi_k = \Phi_k'\alpha_k$ for each $k\in \ZZ_3$. Using the direct sum decomposition given in \eqref{eq:tensor}, we will consider corresponding entries on either side of the equation $\alpha_0\Phi_1 = \Phi_1'\alpha_1$. We will also write $\text{c}(\alpha_k(i,j))$ to denote the constant term of the $(i,j)$th entry of $\alpha_k$. 
        
        Comparing the (1,1) entries of the above equation gives us
            \[\alpha_0(1,1)u + \alpha_0(1,3)z = \alpha_1(1,1)x + \alpha_1(2,1)u.\]
        Equating coefficients of $u$ shows that $\text{c}(\alpha_0(1,1)) = \text{c}(\alpha_1(2,1))$, while the coefficients of $x$ and $z$ give us that $\text{c}(\alpha_0(1,3)) = \text{c}(\alpha_1(1,1))=0$. Similarly, analyzing the other entries, we find that the constant term of each entry of $\alpha_1$ must be zero except possibly the $(1,2),(2,1),$ and $(3,3)$ entries which must satisfy
            \(\text{c}(\alpha_1(1,2)) = \text{c}(\alpha_1(2,1)) = \text{c}(\alpha_1(3,3)).\)
        Repeating this analysis with the second commutative square, $\alpha_1\Phi_2 = \Phi_2'\alpha_2$, we get    
            \(\text{c}(\alpha_1(1,2)) = \zeta \text{c}(\alpha_1(2,1)) = \zeta^2 \text{c}(\alpha_1(3,3)).\)
        The two conclusions are only possible if each constant term is zero. Hence, all entries of $\alpha_1$ have constant term zero contradicting that $\alpha_1$ is invertible.  
    \end{proof}

    Although the tensor product is not symmetric in general, Proposition \ref{thm:otimes_alternate} implies that the functorial properties given below, which we state only for the left component, hold for both components. In particular, the tensor product is additive in both components and associative. We omit the proofs of additivity and associativity as they are straightforward computations utilizing the corresponding properties of $Q$-modules.
    
    \begin{lemma}\,
        \begin{enumerate}[label = (\roman*)]
            \item For any $X, X' \in \matfac{Q}{d}{f}$ and $Y \in \matfac{Q}{d}{g}$, there is a natural isomorphism $(X \oplus X') \wotimes{}  Y \cong (X\wotimes{} Y) \oplus (X' \wotimes{} Y)$.
            \item For any $f,g,h \in Q$ and matrix factorizations $X \in \matfac{Q}{d}{f}, Y \in \matfac{Q}{d}{g},$ and $Z \in \matfac{Q}{d}{h}$, there is an isomorphism $(X \wotimes{} Y) \wotimes{} Z \cong X \wotimes{} (Y \wotimes{} Z)$.
        \end{enumerate}\qed
    \end{lemma}

    \begin{lemma}\label{thm:tensor_with_shifts}
        Let $X \in \matfac{Q}{d}{f}$ and $Y \in \matfac{Q}{d}{g}$. Then
            \begin{enumerate}[label = (\roman*)]
                \item \label{thm:tensor_with_shifts1} $TX \wotimes{} Y \cong T(X \wotimes{} Y) = X \wotimes{} TY$;
                \item \label{thm:tensor_with_shifts2} $T^i X \wotimes{} T^{-i}Y \cong X\wotimes{} Y$ for any $i \in \ZZ_d$.
            \end{enumerate}
    \end{lemma}

    \begin{proof}
        First notice that the equality in \ref{thm:tensor_with_shifts1} follows from the definition and \ref{thm:tensor_with_shifts2} follows from \ref{thm:tensor_with_shifts1}. So, it suffices to prove the isomorphism in \ref{thm:tensor_with_shifts1}.

        Define a map $\alpha: (TX \wotimes{} Y,\Phi) \to (T(X\wotimes{} Y),\Phi')$ given by
        \(\alpha(x\otimes y) = \zeta^{-|y|}x \otimes y\) for homogeneous elements $x \in X$ and $y \in Y$. For homogeneous elements $x \in TX$ and $y \in Y$, $\Phi$ is given by
            \[\Phi(x \otimes y) = \phi(x) \otimes y + \zeta^{|x|-1} x \otimes \psi(y).\]
        Therefore,
            \begin{align*}
                \alpha\Phi(x \otimes y) 
                &= \zeta^{-|y|}\phi(x) \otimes y + \zeta^{|x|-1}\zeta^{-|\psi(y)|} x \otimes \psi(y)\\
                &= \zeta^{-|y|}(\phi(x) \otimes y + \zeta^{|x|} x \otimes \psi(y))\\
                &= \zeta^{-|y|}\Phi'(x\otimes y)\\
                & = \Phi'(\alpha(x\otimes y)).
            \end{align*}
        Hence $\alpha$ is an invertible morphism $TX \wotimes{} Y \to T(X\wotimes{} Y)$.
    \end{proof}

    If $(Q,\mf n)$ is local and $f\in \mf n$ is a nonzerodivisor, the category $\matfac{Q}{d}{f}$ is \textit{Frobenius exact} \cite[Section 2]{tribone2022matrix}. The ring $Q$ is assumed to be regular in \cite[Section 2]{tribone2022matrix} but this assumption is not needed for the Frobenius exact structure. The factorizations of the form $\mc P_i \coloneqq T^i(f,1,1,\dots,1) \in \matfac{Q}{d}{f}$, for $i \in \ZZ_d$, give a complete list of the indecomposable projective objects in $\matfac{Q}{d}{f}$. We will call a matrix factorization $X \in \matfac{Q}{d}{f}$ \textit{projective} if it is isomorphic to a direct sum of copies of $\mc P_i$ for $i \in \ZZ_d$, that is, if $X \cong \bigoplus_{i \in \ZZ_d} \mc P_i^{s_i}$, $s_i \ge 0$.

    \begin{lemma}
        Assume that $(Q,\mf n)$ is local, $f,g\in \mf n$ are nonzerodivisors, and that $\zeta \in Q$ is a distinguished primitive $d^{\textit{th}}$ root of unity. If $P$ is a projective matrix factorization in $\matfac{Q}{d}{f}$, then $P \wotimes{} Y$ is a projective matrix factorization in $\matfac{Q}{d}{f+g}$ for any $Y \in \matfac{Q}{d}{g}$.
    \end{lemma}

    \begin{proof}
        The tensor product $\wotimes{\zeta}$ preserves homotopy equivalences; this is proven in \cite[Proposition 2.5.3]{hopkins2021n}. The conclusion of the lemma follows from this and the fact that an object $P \in \matfac{Q}{d}{h}$ is homotopy equivalent to zero if and only if it is projective (see \cite[Theorem 2.8.3]{hopkins2021n} or \cite[Proposition 2.2.15]{tribone2022matrix} for more details).
    \end{proof}

    Recall that if $(X,\phi)$ is a matrix factorization of a nonzerodivisor $f$ such that $\rank X = n$, then each $X_k$ is a free $Q$-module of rank $n$. In particular, we may view $\phi_k$ as a square matrix with entries in $Q$. Note that if $\rank X = n$ and $\rank Y = m$, then $\rank X\wotimes{} Y = dnm$.
    
    \begin{proposition}\label{thm:det_formula}
        Let $f,g \in Q$ be nonzerodivisors and consider matrix factorizations $X \in \matfac{Q}{d}{f}$ of rank $n$ and $Y \in \matfac{Q}{d}{g}$ of rank $m$. Then $(X\otimes Y,\Phi)\in \matfac{Q}{d}{f+g}$ satisfies $\det \Phi_k = (-1)^{nm(d+1)}(f+g)^{nm}$ for any $k \in \ZZ_d$.
    \end{proposition}

    \begin{proof}
        First, note that it suffices to compute the determinant of $\Phi_1$. Indeed, if $\det\Phi_1$ can be computed for arbitrary $X$ and $Y$, then one can compute $\det\Phi_k$ for all $k \in \ZZ_d$ by applying Lemma \ref{thm:tensor_with_shifts}.

        Our strategy is to convert $\Phi_1$ into a lower triangular form using (not necessarily invertible) row operations. We describe these row operations with respect to the direct sum decompostion \eqref{eq:tensor}. We will also utilize the following notation: If $(X,\phi)$ is a matrix factorization and $r \ge 0$ is an integer, we let $\phi_k^{\{r\}} = \phi_k\phi_{k+1}\cdots \phi_{k+r-1}$ for $r > 0$ and $\phi_k^{\{0\}} = 1_{X_{k-1}}$. Note that these superscripts are not taken modulo $d$. In particular,  $\phi_k^{\{d\}} = f \cdot 1_{X_{k-1}}$ for any $k \in \ZZ_d$.

        First we scale each row of $\Phi_1$ using the diagonal map
            \[\alpha(i,j) = 
                \begin{cases}
                    1_{X_{i-1}} \otimes \psi_2^{\{d-i\}} & \text{ if } i=j\\
                    0 & \text{ else}
                \end{cases}.
            \]
        Note that $\alpha$ is not an isomorphism in general. The resulting product is
    \[
    \alpha\Phi_1 = 
    \begin{psmallmatrix}
        1_{X_0} \otimes \psi_2^{\{d\}} & \phi_1 \otimes \psi_2^{\{d-1\}} & 0 & \cdots & 0 & 0 \\
        0 & \zeta 1_{X_1}\otimes \psi_2^{\{d-1\}} & \phi_2 \otimes \psi_2^{\{d-2\}} & \cdots & 0 & 0\\
        0 & 0 & \zeta^2 1_{X_2}\otimes \psi_2^{\{d-2\}} & \ddots & 0 & 0 \\
        \vdots & \vdots & \vdots & \ddots & \ddots & \vdots \\
        0 & 0 & 0 &  & \zeta^{d-2} 1_{X_{d-2}}\otimes \psi_2^{\{2\}} & \phi_{d-1} \otimes \psi_2^{\{1\}} \\
        \phi_0 \otimes \psi_2^{\{0\}} & 0 & 0 & \cdots & 0 & \zeta^{d-1} 1_{X_{d-1}}\otimes \psi_2^{\{1\}} \\
    \end{psmallmatrix}.
    \]
    Next, using invertible row operations on $\alpha\Phi_1$, we clear above the main diagonal from right to left. We obtain the following lower triangular matrix
    \[\beta = 
    \begin{psmallmatrix}
        c_{d-1}(\phi_1^{\{d\}} \otimes 1_{Y_1} + 1_{X_0} \otimes \psi_2^{\{d\}}) & 0 & 0 & \cdots & 0 & 0 \\
        c_{d-2}\phi_2^{\{d-1\}} \otimes \psi_2^{\{0\}} & \zeta 1_{X_1}\otimes \psi_2^{\{d-1\}} & 0 & \cdots & 0 & 0\\
        c_{d-3}\phi_3^{\{d-2\}} \otimes \psi_2^{\{0\}} & 0 & \zeta^2 1_{X_2}\otimes \psi_2^{\{d-2\}} & & 0 & 0 \\
        \vdots & \vdots & \vdots & \ddots & \vdots & \vdots \\
        c_1\phi_{d-1}^{\{2\}} \otimes \psi_2^{\{0\}} & 0 & 0 &  & \zeta^{d-2} 1_{X_{d-2}}\otimes \psi_2^{\{2\}} & 0 \\
        \phi_0 \otimes \psi_2^{\{0\}} & 0 & 0 & \cdots & 0 & \zeta^{d-1} 1_{X_{d-1}}\otimes \psi_2^{\{1\}} \\
    \end{psmallmatrix}
    \]
    where $c_j = (-1)^j\prod_{k=0}^j \zeta^k$, $1 \leq j \leq d-1$. In particular, we have that $c_{d-1} = (-1)^{d-1}\prod_{k=0}^{d-1}\zeta^k = (-1)^{d-1}(-1)^{d+1} = 1$. It follows that the $\beta(1,1) = (f+g) \cdot 1_{X_0}\otimes 1_{Y_1}$ and therefore 
        \[\det\beta = (f+g)^{nm} \prod_{k = 0}^{d-1}\zeta^{nmk} \cdot \det(\alpha) = (-1)^{nm(d+1)}(f+g)^{nm} \cdot\det(\alpha).\]
    Since $\alpha\Phi_1$ and $\beta$ differ by an isomorphism (with determinant 1), we have 
    \[\det(\alpha) \det(\Phi_1) = \det(\beta) = (-1)^{nm(d+1)}(f+g)^{nm} \cdot\det(\alpha).\] Since $g$ is a nonzerodivisor in $Q$, so is $\det(\alpha)$. The desired conclusion follows by cancelling $\det(\alpha)$.
    \end{proof}
    
\section{Kn\"orrer's direct sum decomposition}\label{sec:knorrer_decomp}

    An important ingredient to the classification of simple hypersurface singularities \cite[Theorem A]{buchweitz1987cohen} and Kn\"orrer's periodicity theorem \cite[Theorem 3.1]{knorrer1987cohen} is a straightforward direct sum decomposition result which says that if $(\phi,\phi)$ is a matrix factorization of an element $f$ and $z$ is an indeterminate, then the matrix factorization $(\phi,\phi) \wotimes{} (z,z)$ of $f+z^2$ decomposes into a direct sum of the form
        \[(\phi,\phi)\wotimes{} (z,z) \cong \left(\begin{pmatrix}
            zI - i\phi & 0 \\
            0 & zI + i\phi
        \end{pmatrix}, \begin{pmatrix}
            zI + i\phi & 0 \\
            0 & zI - i\phi
        \end{pmatrix}\right)\] 
    where $i$ is a square root of $-1$. The key here is that both $(\phi,\phi) \in \matfac{}{2}{f}$ and $(z,z)\in\matfac{}{2}{z^2}$ satisfy $T(\phi,\phi) = (\phi,\phi)$ and $T(z,z) = (z,z)$. This was generalized by Yoshino \cite[Lemma 3.2]{yoshino1998tensor} in the case $d=2$. We give a further extension of these results to the case of $d$-fold matrix factorizations for arbitrary $d \ge 2$. 

    Let $Q$ be a commutative ring and $d \ge 2$ an integer. Fix a primitive $2d^{\textit{th}}$ root of unity, $\omega \in \CC$. In this section we will assume that $Q$ is a $\ZZ[\omega]$-algebra so that $Q$ contains both a $2d^{\textit{th}}$ root of unity and a $d^{\textit{th}}$ root of unity (the images of $\omega$ and $\omega^2$ respectively). We will abuse notation and denote the image of $\omega$ in $Q$ by the same symbol.

    \begin{lemma}\label{thm:roots_of_unity_lemma}
        Let $\omega \in \CC$ be a primitive $2d^{\textit{th}}$ root of unity, and $t$ be an integer such that $t + d$ is even. Then $\sum_{j \in \ZZ_d} \omega^{-j^2 + tj} \neq 0$.  
    \end{lemma}

    \begin{proof}
        We first claim that the expression $\omega^{-j^2 + tj}$ does not depend on the representative chosen for $j$ modulo $d$ (this fact relies on $t+d$ being even). If $j \equiv j' \mod d$, then either $j \equiv j' \mod 2d $, in which case we are done, or $j \equiv j' + d \mod 2d$. In the latter case,      
            \[-j^2 \equiv -(j')^2 + d^2 \mod 2d\] 
        and $tj \equiv tj' + td \mod 2d$. Thus,
            \[-j^2 + tj \equiv -(j')^2 + d^2 + tj' + td \equiv -(j')^2 + tj' + d(d+t) \equiv -(j')^2 + tj' \mod 2d.\]
        
        Now, we claim that $(\sum_{j \in \ZZ_d}\omega^{-j^2 + tj}) (\sum_{l \in \ZZ_d}\omega^{l^2 - tl}) = d$. Performing the change of variable $s = l - j$, we have 
            \begin{align*}
                \left(\sum_{j \in \ZZ_d}\omega^{-j^2 + tj}\right) \left(\sum_{l \in \ZZ_d}\omega^{l^2 - tl}\right)
                &= \left(\sum_{j \in \ZZ_d} \omega^{-j^2 + tj}\right)\left(\sum_{s \in \ZZ_d} \omega^{(s+j)^2 - t(s+j)}\right)\\
                &= \sum_{s \in \ZZ_d} \left(\sum_{j \in \ZZ_d} \omega^{(s^2 - st) + 2sj}\right)
            \end{align*}
         The sum $\sum_{j \in \ZZ_d} \omega^{(s^2 - st) + 2sj} = \omega^{s^2 - st} \sum_{j \in \ZZ_d} \omega^{2sj}$ is equal to $d$ when $s = 0$ in $\ZZ_d$ and $0$ otherwise. This proves the claim and the desired conclusion follows.
    \end{proof}

    Let $\mc C$ be a $Q$-linear category, that is, an additive category such that the abelian group $\hom_{\mc C}(X,Y)$ has a $Q$-module structure for all $X,Y \in \mc C$, and such that composition of morphisms is $Q$-bilinear. In the proposition below, we will consider a pair of commuting morphisms in $\mc C$. By this we mean a pair of morphisms $A,B \in \hom_{\mc C}(X,X)$, for some $X \in \mc C$, such that $AB = BA$. The main result of this section is an application of the following identity among matrices in $\mc C$.
    
    \begin{proposition}\label{thm:symm_decomp_prop}
        Assume that $Q$ is a $\ZZ[\omega,1/d]$-algebra. Let $A$ and $B$ be commuting morphisms in a $Q$-linear category $\mc C$ and set $\Phi$ to be the $d\times d$ block morphism
            \[\Phi = \begin{pmatrix}
                B & A & 0 & \cdots & 0 \\
                0 & \omega^2 B & A & \cdots & 0 \\
                \vdots & \vdots & \ddots & \ddots & \vdots \\
                0 & 0 & 0 & \ddots &A\\
                A & 0 & 0 & \cdots & \omega^{2(d-1)}B
            \end{pmatrix}.\] 
        Then there exist isomorphisms $\alpha_k$, for each $k \in \ZZ_d$, such that
            \[\alpha_{k-1}\Phi \alpha_k^{-1} = \begin{pmatrix}
                A - \omega^{2k-1}B & 0 & \dots & 0 \\
                0 & A - \omega^{2k+1}B  & \dots & 0 \\
                \vdots & \vdots & \ddots & \vdots \\
                0 & 0 & \dots & A - \omega^{2k-3}B 
                \end{pmatrix}.\]
    \end{proposition}

    \begin{proof}
        To simplify the notation, we introduce the polynomial $\mf p(m) = -m^2 + dm$, $m \in \ZZ$. For each $k \in \ZZ_d$, we define a map $\alpha_k : X^d \to X^d$, where $X$ is the common domain/codomain of $A$ and $B$, by the following rule: for any $i,j \in \ZZ_d$, the $(i,j)$-component of $\alpha_k$ is defined to be $\alpha_k (i, j) = \omega^{\mf p(j-i-k)} 1_X$. As in the proof of Lemma \ref{thm:roots_of_unity_lemma}, this formula does not depend on the representatives chosen for $i, j, k$. For each $k \in \ZZ_d$, let $\Phi_k'$ denote the $d \times d$ block diagonal morphism in the statement of the proposition. We will show that $\alpha_{k-1} \Phi_k = \Phi_k' \alpha_k$ by computing the $(i, j)$-component of each composition. 
        
        Fix $k \in \ZZ_d$. The only non-zero entries of the $j$th column of $\Phi$ are $\Phi(j-1, j) = A$ and $\Phi(j, j) = \omega^{2(j-1)} B$. Therefore,
            \begin{align*}
                (\alpha_{k-1} \Phi)(i, j) 
                &= \alpha_{k-1} (i,j-1) A + \alpha_{k-1}(i,j) \omega^{2(j-1)}B\\
                &= \omega^{\mf p(j-i-k)}A + \omega^{\mf p(j-i-k+1) + 2(j-1)} B\\
                &= \omega^{\mf p(j-i-k)}A + \omega^{\mf p(j-i-k) + d + 2k + 2(i-1) -1}B
            \end{align*}

    Similarly, the only non-zero entry in the $i$th row of $\Phi_k'$ is $A - \omega^{2k + 2(i-1) -1} B$, so 
        \[(\Phi_k' \alpha_k) (i, j) = \omega^{\mf p(j-i-k)}A - \omega^{\mf p(j-i-k) + 2k + 2(i-1) - 1}B.\] 
    Since $\omega^{d} = -1$, it follows that $(\alpha_{k-1}\Phi)(i,j) = (\Phi_k' \alpha_k)(i,j)$ for all $i,j,k \in \ZZ_d$.

    It remains to show that $\alpha_k$ is an isomorphism for each $k \in \ZZ_d$. It suffices to prove this in the case that $\alpha_k$ is the $d\times d$ matrix with entries in $\ZZ[\omega,1/d]$ given by $\alpha_k(i,j) = \omega^{\mf p(j-i-k)}$, $i,j \in \ZZ_d$. Fixing $k \in \ZZ_d$, we observe that $\alpha_k$ is a \textit{circulant matrix}, as $\alpha_k (i,j)$ is a function of $j - i$, i.e. if $j - i = j' - i'$ so that $\alpha_k(i,j)$ and $\alpha_k(i', j')$ lie on the same diagonal (modulo $d$) of $\alpha_k$, then $\alpha_k (i, j) = \alpha_k (i', j')$. Since $\alpha_k$ is circulant, the determinant is given by
        \begin{align*}
            \det\alpha_k 
            &= \prod_{s=1}^d (\alpha_k (1, 1) + \alpha_k (1, 2) \omega^{2s} + \alpha_k (1, 3) \omega^{4s} + \dots + \alpha_k(1, d) \omega^{2(d-1)s})\\
            &= \prod_{s = 1}^d \left(\sum_{j = 1}^d \omega^{\mf{p}(j-1-k) + 2s(j-1)}\right).
        \end{align*}
    By Lemma \ref{thm:roots_of_unity_lemma}, each of these sums is a unit in $\ZZ[\omega,1/d]$; take $t = 2 + 2k + 2s + d$. Hence $\det\alpha_k$ is a unit implying that $\alpha_k$ is invertible.
\end{proof}

   \begin{remark}
       Proposition \ref{thm:symm_decomp_prop} can be viewed as a matrix analogy to the polynomial factorization of $x^d + y^d$ into linear factors (see Example \ref{ex:knorrer_decomp_example}). Indeed, this explains the presence of the $2d^{\textit{th}}$ roots of unity which allow us to unify the even and odd case into one statement. In particular, when $d$ is odd, the condition that $Q$ is a $\ZZ[\omega]$-algebra is automatically satisfied by our usual assumption: if $\zeta \in \CC$ is a primitive $d^{\textit{th}}$ root of unity, then $\omega = - \zeta$ is a primitive $2d^{\textit{th}}$ root of unity. Thus, $\ZZ[\zeta] = \ZZ[\omega]$ and therefore $Q$ is a $\ZZ[\zeta]$-algebra if and only if it is a $\ZZ[\omega]$-algebra. This is, however, not true when $d$ is even. In this case, the condition that $Q$ is a $\ZZ[\omega]$-algebra is strictly stronger. 
   \end{remark}

   The main theorem of this section is a direct application of Proposition \ref{thm:symm_decomp_prop}.

    \begin{theorem}\label{thm:symm_decomposition}
        Let $f,g \in Q$ and assume that $Q$ is a $\ZZ[\omega,1/d]$-algebra. Suppose $X \in \matfac{Q}{d}{f}$ and $Y \in \matfac{Q}{d}{g}$ satisfy $TX = X$ and $TY = Y$. Then there exists $Z \in \matfac{Q}{d}{f+g}$ such that
        \[X \wotimes{\zeta} Y \cong \bigoplus_{i \in \ZZ_d} T^iZ\]
        where the tensor product is taken with respect to the $d^{\textit{th}}$ root of unity $\zeta = \omega^2$.
    \end{theorem}

    \begin{proof}
        If $X = (\phi,\phi,\dots,\phi) \in \matfac{Q}{d}{f}$ and $Y = (\psi,\psi,\dots,\psi) \in \matfac{Q}{d}{g}$, then set $B = 1 \otimes \psi$ and $A = \phi \otimes 1$ and apply Proposition \ref{thm:symm_decomp_prop}. The factorization $Z$ is given by
            \[Z = \left(\phi\otimes 1 - \omega (1 \otimes \psi), \phi\otimes 1 - \omega^3(1 \otimes \psi),\dots,\phi\otimes 1 - \omega^{2d-1}(1 \otimes \psi) \right) \in \matfac{Q}{d}{f+g}.\]
    \end{proof}

     If the factorizations $X$ and $Y$ are only isomorphic to their shifts (not necessarily equal), then a similar result is possible with additional assumptions.
    
    \begin{corollary}\label{thm:symm_decomposition_cor}
        Assume $(Q,\mf n, \Bbbk)$ is a complete regular local ring with $\Bbbk$ algebraically closed of characteristic not dividing $d$ and fix non-zero elements $f,g\in \mf n^2$. Suppose $X \in \matfac{Q}{d}{f}$ and $Y \in \matfac{Q}{d}{g}$ satisfy $TX \cong X$ and $TY \cong Y$. Then there exists $Z \in \matfac{Q}{d}{f+g}$ such that
            \[X \wotimes{\zeta} Y \cong \bigoplus_{i \in \ZZ_d} T^iZ\]
        where the tensor product is taken with respect to the $d^{\textit{th}}$ root of unity $\zeta = \omega^2$.
    \end{corollary} 

    \begin{proof}
        The additional assumptions allow us to apply \cite[Corollary 4.3]{leuschke2023branched}, that is, we may assume that $X = TX$ and $Y = TY$. The corollary now follows from Theorem \ref{thm:symm_decomposition}.
    \end{proof}

    In the case $d=2$, Theorem \ref{thm:symm_decomposition} and Corollary \ref{thm:symm_decomposition_cor} recover \cite[Lemma 3.2]{yoshino1998tensor} and by extension, \cite[Proposition 2.7]{knorrer1987cohen}.

    \begin{example}\label{ex:knorrer_decomp_example}
        Let $Q = \Bbbk[x,y]$ where $\Bbbk$ is a domain. 

        \begin{enumerate}[label = (\roman*)]
            \item Assume that $\Bbbk$ contains a primitive fourth root of unity, $i \in \Bbbk$. Consider the rank one factorizations $X = (x,x) \in \matfac{Q}{2}{x^2}$, $Y = (y,y) \in \matfac{Q}{2}{y^2}$, and $Z = (x-iy,x-i^3y) = (x-iy,x+iy) \in \matfac{Q}{2}{x^2+y^2}$. Then Theorem \ref{thm:symm_decomposition} implies that $X \wotimes{} Y \cong Z \oplus TZ$, that is,
                \[\left(\begin{pmatrix}
                    y & x \\
                    x & -y
                \end{pmatrix},\begin{pmatrix}
                    y & x \\
                    x & -y
                \end{pmatrix}\right) \cong \left(\begin{pmatrix}
                    x-iy & 0 \\
                    0 & x+iy
                \end{pmatrix},\begin{pmatrix}
                    x+iy & 0 \\
                    0 & x-iy
                \end{pmatrix}\right).\]
                
            \item Assume that $\Bbbk$ contains a primitive third root of unity, $\zeta \in \Bbbk$. Then $\omega = - \zeta$ is a primitive $6^{\textit{th}}$ root of unity. Consider the rank one factorizations $X = (x,x,x) \in \matfac{Q}{3}{x^3}$, $Y = (y,y,y) \in \matfac{Q}{3}{y^3}$, and
                \[Z = (x - \omega y, x - \omega^3 y, x - \omega^5 y) = (x+\zeta y,x+ y, x+\zeta^2y) \in \matfac{Q}{3}{x^3+y^3}.\]
        In this case, Theorem \ref{thm:symm_decomposition} implies that
            \( X \wotimes{\zeta} Y \cong Z \oplus TZ \oplus T^2Z. \)
        \end{enumerate}
    \end{example}

\section{Direct sum decompositions of tensor products}\label{sec:disjoint_vars}

    Following the work of \cite{yoshino1998tensor}, we will consider tensor products of matrix factorizations of polynomials $f$ and $g$ which are written in \textit{disjoint} sets of variables. In this setting, we are able to prove several decomposition results, most of which are direct generalizations of the results shown in \cite{yoshino1998tensor}. We begin by defining the key functor used in this section.

    \begin{definition}\label{def:reduction_mod_I}
        Let $Q$ be a commutative ring, $d \ge 2$, and $f \in Q$. If $(X,\phi) \in \matfac{Q}{d}{f}$, and $I$ is an ideal of $Q$, then we let $X_I$ denote the matrix factorization
            \[X_I \coloneqq (X \otimes_Q Q/I, \phi \otimes_Q 1_{Q/I}) \in \matfac{Q/I}{d}{f+I}.\]
        Notice that $\rank X = \rank X_I$, if both ranks are well-defined. Furthermore, \textit{reduction modulo} $I$ forms a functor 
                \[( - )_I : \matfac{Q}{d}{f} \to \matfac{Q/I}{d}{f+I}.\]
    \end{definition}
    
    We will adopt the following notation throughout the rest of this section.

    \begin{notation}\label{notation}
        Let $d \ge 2$ be an integer and let $\Bbbk$ be a field containing a primitive $d^{\textit{th}}$ root of unity, $\zeta \in \Bbbk$. Fix the power series rings $S_1 = \Bbbk\llbracket x_1,x_2,\dots,x_a\rrbracket$, $S_2 = \Bbbk\llbracket y_1,y_2,\dots,y_b \rrbracket$, and $S = \Bbbk\llbracket x_1,\dots,x_a,y_1,\dots,y_b\rrbracket$ and fix non-invertible non-zero elements $f \in S_1$ and $g \in S_2$. By tensoring with $S$ we obtain faithful embeddings $\matfac{S_1}{d}{f} \hookrightarrow \matfac{S}{d}{f}$ and $\matfac{S_2}{d}{g} \hookrightarrow \matfac{S}{d}{g}$. We will identify a factorization $X \in \matfac{S_1}{d}{f}$ with its image in $\matfac{S}{d}{f}$, and similarly for $Y \in \matfac{S_2}{d}{g}$, in order to form the tensor product $X \wotimes{} Y \in \matfac{S}{d}{f+g}$. Together with reduction modulo $(x) \coloneqq S(x_1,x_2,\dots,x_a)$ and reduction modulo $(y) \coloneqq S(y_1,y_2,\dots,y_b)$, we get the following diagram of functors:
            \[
            \begin{tikzcd}
                \matfac{S_1}{d}{f} \rar[shift left=2pt]{\wotimes{}} &\matfac{S}{d}{f+g}  \lar[shift left=2pt]{(-)_y} \rar[swap,shift right=2pt]{(-)_x}  &\matfac{S_2}{d}{g} \lar[swap, shift right=2pt]{\wotimes{}}
            \end{tikzcd}.
            \]
    \end{notation}

    The interaction between these functors will be key going forward. Recall that a matrix factorization $(X,\phi)$, over a local ring $(Q,\mf n)$, is said to be \textit{reduced} if $\image \phi_k \subseteq \mf n X_{k-1}$ for each $k \in \ZZ_d$. In other words, after choosing bases, the entries of the matrix representing $\phi_k, k \in \ZZ_d$, live in the maximal ideal of $Q$.

    \begin{proposition}\label{thm:tensor_modx_mody} Let $(X,\phi) \in \matfac{S_1}{d}{f}$ and $(Y,\psi) \in \matfac{S_2}{d}{g}$.
        \begin{enumerate}[label = (\roman*)]
            \item \label{thm:tensor_modx_mody1} Suppose $X \in \matfac{S_1}{d}{f}$ is reduced of rank $n$. Then there is an isomorphism in $\matfac{S_2}{d}{g}$ 
                \[(X\wotimes{} Y)_x \cong \bigoplus_{i=1}^{d} (\zeta^{i-1}T^{1-i}Y)^n.\]
            \item \label{thm:tensor_modx_mody2} Suppose $Y \in \matfac{S_2}{d}{g}$ is reduced of rank $m$. Then there is an isomorphism in $\matfac{S_1}{d}{f}$
                \[(X\wotimes{} Y)_y \cong \bigoplus_{i=1}^{d} (\zeta^{1-i}T^{1-i}X)^m.\]
        \end{enumerate}
    \end{proposition}

    \begin{proof}
        Since $(X,\phi)$ is reduced, the entries of $\phi$ live in the ideal $(x)$. This implies that $\phi \otimes 1_{S/(x)} = 0$. On the other hand, since $\psi$ has entries in $S_2$, we may identify $\psi_k \otimes 1_{S/(x)}$ with $\psi_k$ and so, for any $j,k \in \ZZ_d$, we have
            \[(1_{X_j} \otimes \psi_k) \otimes_S 1_{S/(x)} \cong \psi_k^{\oplus n}.\]
        Considering \eqref{eq:tensor}, we obtain an isomorphism
                \((X\wotimes{} Y)_x \cong \bigoplus_{i=1}^d (\zeta^{i-1}T^{1-i}Y)^n.\)
        
        The second isomorphism follows similarly by considering the isomorphism $X \wotimes{\zeta} Y \cong Y \wotimes{\zeta^{-1}} X$ given in Proposition \ref{thm:otimes_alternate}.
    \end{proof}

    The powers of $\zeta$ can also be omitted from these isomorphisms by Lemma \ref{thm:scaled_by_units}. In particular, this lemma is the $d$-fold analogue of \cite[Lemma 2.9]{yoshino1998tensor} and \cite[Lemma 2.5]{knorrer1987cohen}.

    In order to investigate the number of indecomposable direct summands of $X \wotimes{} Y$, we will work directly with idempotents in the endomorphism ring of $X \wotimes{} Y$. The next lemma shows how the functors $(-)_x$ and $(-)_y$ interact with morphisms of this form and more generally how they interact with morphisms between tensor products.

    \begin{lemma}\label{thm:morphisms_mod_x/y}\,
        \begin{enumerate}[label = (\roman*)]
            \item \label{thm:morphisms_mod_x} Assume $X \in \matfac{S_1}{d}{f}$ is reduced of rank $n$, $Y,Y' \in \matfac{S_2}{d}{g}$, and $\alpha \in \hom(X\wotimes{} Y, X \wotimes{} Y')$. Then, for $i,j \in \ZZ_d$, the $(i,j)$-component of the induced morphism
                \[\alpha_x: \bigoplus_{j=1}^d (\zeta^{j-1}T^{1-j}Y)^n \to \bigoplus_{i=1}^d (\zeta^{i-1}T^{1-i}Y')^n\] 
            is given by the morphism
                \[\alpha_x(i,j) = (\alpha_0(i,j)_x,\alpha_1(i,j)_x,\dots,\alpha_{d-1}(i,j)_x) = (\alpha_k(i,j)_x)_{k=0}^{d-1}\]
            in $\matfac{S_2}{d}{g}$.

            \item \label{thm:morphisms_mod_y} Assume $Y \in \matfac{S_2}{d}{g}$ is reduced of rank $m$, $X,X' \in \matfac{S_1}{d}{f}$, and $\beta \in \hom(X\wotimes{} Y, X' \wotimes{} Y)$. Then, for $i,j \in \ZZ_d$, the $(i,j)$-component of the induced morphism
                \[\beta_y: \bigoplus_{j=1}^d (T^{1-j}X)^m \to \bigoplus_{i=1}^d (T^{1-i}X')^m\]
            is given by the morphism
                \[\beta_y(i,j) = (\beta_k(2-i+k,2-j+k)_y)_{k=0}^{d-1}\]
            in $\matfac{S_1}{d}{f}$.
        \end{enumerate}
    \end{lemma}

    \begin{proof}
        The first statement follows directly from \eqref{eq:tensor} and Proposition \ref{thm:tensor_modx_mody}. To see that \ref{thm:morphisms_mod_y} holds, we introduce notation. Let $C$ be a $d \times d$ permutation matrix corresponding to a permutation $\sigma$ of $\{1,2,\dots,d\}$. We allow $C$ to act on direct sums of modules in the following way: if $H = \bigoplus_{i=1}^d H_i$, then we set $CH = \bigoplus_{i=1}^d H_{\sigma(i)}$ and view $C$ as a homomorphism $C: H \to CH$ given by $(h_1,h_2,\dots,h_d) \mapsto (h_{\sigma(1)},h_{\sigma(2)},\dots,h_{\sigma(d)})$.

        We will consider two such permutation matrices to justify the formula for $\beta_y$:
            \[C = \begin{pmatrix}
                0 & 0 & \cdots  & 0 & 1\\
                1 & 0 & \ddots  & 0 & 0\\
                0 & 1 & \ddots & 0 & 0\\
                \vdots & \vdots & \ddots & \vdots & \vdots\\
                0 & 0 & \cdots & 1 & 0
            \end{pmatrix} \quad \text{ and } \quad B =\begin{pmatrix}
                1 & 0 & \cdots & 0 & 0\\
                0 & 0 & \ddots & 0 & 1\\
                0 & 0 & \ddots & 1 & 0\\
                \vdots & \vdots & \ddots & \ddots\\
                0 & 1 & \cdots & 0 & 0
            \end{pmatrix}.\]
        If $h: \bigoplus_{j=1}^{d}H_j \to \bigoplus_{i=1}^d H_i'$ is a homomorphism, then conjugation by $C^k$, $k \in \ZZ_d$, cycles the components of $h$; for any $i,j \in \ZZ_d$,
            \[(C^{-k} h C^k)(i,j) = h(i+k,j+k): H_{j+k} \to H_{i+k}'.\]
        On the other hand, noting that $B = B^{-1}$, conjugation by $B$ gives
            \[(BhB)(i,j) = h(2-i,2-j):H_{2-j} \to H_{2-i}'.\]
        
        With this set up, we write $\beta_y$ with respect to the direct sum decomposition \eqref{eq:tensor}, and apply these permutations to the components of $\beta_y$. One can then check directly that $\beta_y$ induces the morphism 
            \[\left(B(\beta_0)_yB,BC^{-1}(\beta_1)_yCB,\dots,BC^{-(d-1)}(\beta_{d-1})_yC^{d-1}B\right)\]
        from $\bigoplus_{j=1}^d (T^{1-j}X)^m \to \bigoplus_{i=1}^d (T^{1-i}X')^m$. For any $i,j \in \ZZ_d$, the $(i,j)$-component of this morphism is
            \begin{align*}
                \left((BC^{-k}(\beta_k)_yC^kB)(i,j)\right)_{k=0}^{d-1} &= \left((C^{-k}(\beta_k)_yC^k)(2-i,2-j)\right)_{k=0}^{d-1}\\
                &= \left((\beta_k)_y(2-i+k,2-j+k)\right)_{k=0}^{d-1}\\
                &= \left(\beta_k(2-i+k,2-j+k)_y\right)_{k=0}^{d-1}
            \end{align*}
        which is the formula claimed in \ref{thm:morphisms_mod_y}.
            
    \end{proof}

\subsection{Reduced matrix factorizations}

    In this section, we will restrict our attention to tensor products of reduced matrix factorizations only. 

    \begin{notation}\label{notation2}
        In addition to the notation given in \eqref{notation}, we will assume that $X \in \matfac{S_1}{d}{f}$ and $Y \in \matfac{S_2}{d}{g}$ are indecomposable reduced matrix factorizations of $f$ and $g$ respectively. We denote the number of indecomposable summands (counted with multiplicity) in the direct sum decomposition of $X \wotimes{} Y$ by $\#(X\wotimes{} Y)$. This quantity is well defined since the Krull-Schmidt Theorem holds in $\matfac{S}{d}{f+g}$ when $S$ is complete as it is in this section (see \cite[Section 3]{tribone2022matrix}). We also set $r = \gcd(n,m)$ where $n = \rank X$ and $m = \rank Y$. 
    \end{notation}

    Our first result is a direct extension of \cite[Theorem 3.3]{yoshino1998tensor}.
    
    \begin{theorem}\label{thm:dr_summands}
        The tensor product $X \wotimes{} Y$ has at most $dr$ indecomposable summands.
    \end{theorem}

    \begin{proof}
        Let $Z\in\matfac{S}{d}{f+g}$ be a summand of $X\wotimes{} Y$. By Lemma \ref{thm:tensor_modx_mody}, $Z_y$ is isomorphic to a summand of $\bigoplus_{i \in \ZZ_d} (T^i X)^m$ (we omit the powers of $\zeta$ using Lemma \ref{thm:scaled_by_units}). Since $T^{i} X$ is indecomposable for each $i\in \ZZ_d$, the Krull-Schmidt property of $\matfac{S_1}{d}{f}$ implies that $Z_y \cong \bigoplus_{i \in \ZZ_d} (T^i X)^{t_i}$ for some integers $0 \leq t_i \leq m$. Hence, the rank of $Z_y$, which is the same as the rank of $Z$, is $(t_0+\dots + t_{d-1})n$. Similarly, reduction modulo $(x)$ gives us that the rank of $Z$ is equal to $(s_0 + \dots + s_{d-1})m$ for some integers $0\leq s_i \leq n$. Together, these imply that $(t_0+\dots+t_{d-1})n = (s_0+\dots+s_{d-1})m$ which must be at least $\lcm(n,m) = nm/r$. Since $X\wotimes{} Y$ is of rank $dnm$ and we have just shown that any summand must be of rank at least $nm/r$, it follows that $\#(X\wotimes{} Y) \leq dr$.
    \end{proof}

    This upper bound can be improved if the factorizations are highly non-symmetric (cf. \cite[Theorem 3.4]{yoshino1998tensor}).
   
    \begin{theorem}\label{thm:totally_non_symmetric_tensor_product}
        Suppose $X\not\cong T^iX$ and $Y \not\cong T^jY$ for all $i,j \neq 0$. Then $\#(X\wotimes{} Y) \leq r$.
    \end{theorem}

    The proof of Theorem \ref{thm:totally_non_symmetric_tensor_product} will require some preparation.
    
    \subsubsection{The radical of a Krull-Schmidt category}
    
    Let $\mc C$ be a \textit{Krull-Schmidt category}, that is, an additive category such that every object decomposes into a finite sum of objects which have local endomorphism rings. Given a pair of objects $X,Y \in \mc C$, we will utilize the notion of the \textit{radical}, denoted $\rad_{\mc C}(X,Y)$. This is an ideal of $\mc C$ and is the category theoretic generalization of the Jacobson radical for rings. Indeed, $\rad_{\mc C}$ is the unique ideal of $\mc C$ which satisfies $\rad_{\mc C}(X,X) = \rad(\End_{\mc C}(X))$, for all $X \in \mc C$, where $\rad(-)$ on the right hand side denotes the usual Jacobson radical \cite[Proposition 2.9]{krause2015krull}. For our purposes, it will suffice to use the following equivalent definition of the radical.

    \begin{definition}\cite[Corollary 2.10]{krause2015krull}
        For objects $X,Y \in \mc C$, the we define the \textit{radical}, $\rad_{\mc C}(X,Y)$, to be the collection of morphisms $\alpha \in \hom_{\mc C}(X,Y)$ such that $1_Z - \sigma \alpha \tau$ is an isomorphism for all $\sigma \in \hom_{\mc C}(Y,Z)$ and $\tau \in \hom_{\mc C}(Z,X)$.
    \end{definition}
     We recall several useful facts about $\rad = \rad_{\mc C}$ which follow from \cite[Section 4]{krause2015krull} or \cite[Example 2.1.25]{krause2021homological}. Note that, while working with objects in the category $\mc C$, our indices are not taken modulo $d$.

    \begin{lemma}\label{thm:rad_properties}
        Let $X$ and $Y$ be objects in a Krull-Schmidt category $\mc C$.
        \begin{enumerate}[label = (\roman*)]
            \item \label{thm:rad_properties_1} Suppose $X=X_1^{n_1}\oplus \cdots \oplus X_t^{n_t}$ and $Y = Y_1^{m_1}\oplus \cdots \oplus Y_s^{m_s}$ for indecomposable objects $X_i,Y_j$, and positive integers $n_i,m_j$, for $1\leq i \leq t, 1\leq j \leq s$. Then $\rad(X,Y) \cong \bigoplus_{i,j}\rad(X_i^{n_i},Y_j^{m_j})$.
            \item \label{thm:rad_properties_2} Suppose that $X$ and $Y$ are indecomposable objects such that $X \not\cong Y$. Then $\rad(X^n,Y^m) = \hom(X^n,Y^m)$ for any $n,m\ge 1$. \qed
        \end{enumerate}
    \end{lemma}

    \begin{lemma}\label{thm:idempotentlemma}
        Let $X_1,\dots,X_m$ be indecomposable objects in a Krull-Schmidt category $\mathcal C$ and set $X = X_1^{n_1}\oplus \cdots \oplus X_m^{n_m}$ for positive integers $n_1,\dots,n_m$. Let $e = (e(i,j))_{i,j}$ be an idempotent in $\End(X)$, where $e(i,j):X_j^{n_j}\to X_i^{n_i}$. If $e(i,k)e(k,j) \in \rad(X_j^{n_j},X_i^{n_i})$ for all $i,j,k$ where either $i\neq k$ or $j\neq k$, then there exist idempotents $e_i \in \End(X_i^{n_i})$, $1\leq i \leq m$, such that $e(i,i) - e_i \in \rad\End(X_i^{n_i})$ and $e(X) \cong \bigoplus_{i=1}^m e_i(X_i^{n_i})$, where $e(X)$ denotes the direct summand of $X$ given by the idempotent $e$.
    \end{lemma}

    \begin{proof}
        Since $e = e^2$, for each $1\leq i\leq m$, we have that
            \(e(i,i) = e^2(i,i) = \sum_{k=1}^m e(i,k)e(k,i).\)
        Since $e(i,k)e(k,i) \in \rad\End(X_i^{n_i})$ for all $i\neq k$ by assumption, we have that
            \[e(i,i) \equiv e(i,i)^2 \mod \rad\End(X_i^{n_i}).\] 
        Since $\mc C$ is a Krull-Schmidt category, we may lift $e(i,i)$ to an idempotent in $\End(X_i^{n_i})$, that is, there exists an idempotent $e_i \in \End(X_i^{n_i})$ such that $e(i,i) - e_i \in \rad(X_i^{n_i},X_i^{n_i})$. Let $\epsilon = \bigoplus_{i=1}^m e_i \in \End(X)$ and set $\gamma = e - \epsilon$. We claim that $\gamma^2 \in \rad\End(X)$. To see this, first notice that for each $i$,
            \[\gamma^2(i,i) = \sum_{k\neq i}e(i,k)e(k,i) + (e(i,i)-e_i)^2 \in \rad(X_i^{n_i},X_i^{n_i}).\]
        If $i\neq j$, then
            \[\gamma^2(i,j) = (e(i,i)-e_i) e(i,j) + e(i,j) (e(j,j) - e_j) + \sum_{k\neq i,k\neq j} e(i,k)e(k,j).\]
        Since $\rad_{\mc C}$ is an ideal in the category $\mc C$, the fact that $e(i,i)-e_i \in \rad(X_i^{n_i},X_i^{n_i})$ implies that $(e(i,i)-e_i)e(i,j) \in \rad(X_j^{n_j},X_i^{n_i})$. Similarly, we have that $e(i,j) (e(j,j) - e_j) \in \rad(X_j^{n_j},X_i^{n_i})$. Since the rest of the terms are in $\rad(X_j^{n_j},X_i^{n_i})$ by assumption, we have that $\gamma^2(i,j) \in \rad(X_j^{n_j},X_i^{n_i})$ for all $i,j$. Lemma \ref{thm:rad_properties}\ref{thm:rad_properties_1} then implies that $\gamma^2 \in \rad\End(X)$ as claimed.

        We can now finish the proof of the lemma. Since $\gamma^2 \in \rad\End(X)$, we have that $1_X - \gamma^2 = (1_X+\gamma)(1_X-\gamma)$ is a unit in $\End(X)$. Hence, both $(1_X-\gamma)$ and $(1_X+\gamma)$ are units in $\End(X)$. Note that, since $e$ and $\epsilon$ are idempotents and $\gamma = e - \epsilon$, we have
        \[(1_X+\gamma)\epsilon = \epsilon^2 + \gamma\epsilon = (\epsilon + \gamma)\epsilon = e\epsilon = e(e-\gamma) = e-e\gamma = e(1_X-\gamma).\] That is, $e$ and $\epsilon$ are equivalent idempotents. It follows that $e(X) \cong \epsilon(X) = \bigoplus_{i=1}^m e_i(X_i^{n_i})$.
    \end{proof}

    If, in addition, the indecomposable objects $X_1,X_2,\dots,X_m$ are pairwise non-isomorphic, then Lemma \ref{thm:rad_properties}\ref{thm:rad_properties_2} implies that the assumptions of the lemma are automatically satisfied. With this, we are able to prove the Theorem \ref{thm:totally_non_symmetric_tensor_product}. Our approach is a direct generalization of \cite[Theorem 3.4]{yoshino1998tensor}.

    \begin{proof}[Proof of Theorem \ref{thm:totally_non_symmetric_tensor_product}]
        Suppose $\epsilon \in \End(X\otimes Y)$ is an idempotent. Applying Lemma \ref{thm:tensor_modx_mody}\ref{thm:tensor_modx_mody1}, we obtain an idempotent $\epsilon_x \in \hom(\bigoplus_{j=1}^d(\zeta^{j-1}T^{1-j}Y))$. By assumption, the indecomposable matrix factorizations $\zeta^{j-1}T^{1-j}Y \cong T^{1-j}Y, j \in \ZZ_d$, are pairwise non-isomorphic. From Lemma \ref{thm:idempotentlemma}, we obtain idempotents $e_j \in \End(\zeta^{j-1}T^{1-j}Y)^n$, $j \in \ZZ_d$, such that \[\epsilon_x(j,j) - e_j \in \rad \End(\zeta^{j-1}T^{1-j}Y)^n)\] and $\epsilon_x((X\wotimes{} Y)_x) \cong \bigoplus_{j=1}^d e_j(\zeta^{j-1}T^{1-j}Y)^n)$. Since $T^{1-j}Y$ is indecomposable for each $j \in \ZZ_d$, the Krull-Schmidt theorem implies
            \begin{equation}
                \bigoplus_{j=1}^d e_j(\zeta^{j-1}T^{1-j}Y)^n) \cong \bigoplus_{j=1}^d (\zeta^{j-1}T^{1-j}Y)^{t_j}
            \end{equation} for some integers $0 \leq t_j \leq n$, $j \in \ZZ_d$.
        
        Fix $j \in \ZZ_d$. We may apply Lemma \ref{thm:idempotent_normal_form} to obtain an equivalent idempotent
                \[e_j \sim \begin{pmatrix}
                    1_{t_j} & 0 \\
                    0 & 0
                \end{pmatrix}\]
        where $1_{t_j}$ denotes the identity morphism on the first summand of
            \[(\zeta^{j-1}T^{1-j}Y)^n = (\zeta^{j-1}T^{1-j}Y)^{t_j}\oplus (\zeta^{j-1}T^{1-j}Y)^{n-t_j}.\] 
        Then, since $\epsilon_x(j,j) - e_j \in \rad \End(\zeta^{j-1}T^{1-j}Y)^n)$, we use Lemma \ref{thm:equiv_morphisms} to clear columns and rows and obtain
            \[\epsilon_x(j,j) \sim \begin{pmatrix}
                        1_{t_j} & 0 \\
                        0 & A_j
                    \end{pmatrix}\]
        for some $A_j \in \rad\End((\zeta^{j-1}T^{1-j}Y)^{n-t_j})$. Putting this together for all $j \in \ZZ_d$, and applying Lemma \ref{thm:equiv_morphisms} again, we obtain
            \[\epsilon_x \sim \begin{pmatrix}
                1 & 0 \\
                0 & B
                \end{pmatrix}\] 
        where $1$ is the identity morphism on $\bigoplus_{j=1}^d (\zeta^{j-1}T^{1-j}Y)^{t_j}$ and $B$ is a morphism in $$\End(\bigoplus_{j=1}^d (\zeta^{j-1}T^{1-j}Y)^{n-t_j})$$ satisfying $B(j,j) = A_j$ for each $j \in \ZZ_d$. This implies that there is an isomorphism of free $Q$-modules, $\image \epsilon_x \cong \bigoplus_{j=1}^d (\zeta^{j-1}T^{1-j}Y)^{t_j} \oplus \image B$. However, we know that $\epsilon_x((X\wotimes{} Y)_x) \cong \bigoplus_{j=1}^d (\zeta^{j-1}T^{1-j}Y)^{t_j}$ and therefore, counting ranks, $B$ must be zero. Hence $B(j,j) = A_j = 0$ for each $j \in \ZZ_d$. In other words,
            \[\epsilon_x(j,j) \sim \begin{pmatrix}
                1_{t_j} & 0 \\
                0 & 0 
                \end{pmatrix} \sim e_j\]
        for each $j \in \ZZ_d$. Now, recall that the $(i,j)$-component of $\epsilon_x$ is described by Lemma \ref{thm:morphisms_mod_x/y}\ref{thm:morphisms_mod_x}. In particular, for $j \in \ZZ_d$,
            \(\epsilon_x(j,j) = (\epsilon_0(j,j)_x,\epsilon_1(j,j)_x,\dots,\epsilon_{d-1}(j,j)_x).\)
        It follows that $\epsilon_i(j,j)_x \sim 
            \begin{pmatrix}
            I_{nt_j} & 0 \\
            0 & 0
            \end{pmatrix}$,
        for each $i \in \ZZ_d$, where $I_{nt_j}$ is an identity matrix of rank $nt_j$. This allows us to make the following rank count:
            \begin{equation}\label{eq:rank_count_modx}
                \rank_\Bbbk(\epsilon_i(j,j) \otimes_S \Bbbk) = \rank_\Bbbk(\epsilon_i(j,j)_x \otimes_S \Bbbk) = nt_j \text{ for each } i,j \in \ZZ_d.
            \end{equation} 

        Applying the same steps to $\epsilon_y$, this time using Lemma \ref{thm:morphisms_mod_x/y}\ref{thm:morphisms_mod_y}, we find that, for each $j \in \ZZ_d$, there exists an integer $0 \leq s_j \leq m$ such that
            \[\epsilon_i(2-j+i,2-j+i)_y \sim \begin{pmatrix}
                I_{ms_j} & 0 \\
                0 & 0
            \end{pmatrix}\]
        for all $i \in \ZZ_d$. In other words,
            \begin{equation}\label{eq:rank_count_mody}
                \rank_\Bbbk(\epsilon_i(2-j+i,2-j+i) \otimes_S \Bbbk) = ms_j \text{ for each } i,j \in \ZZ_d.
            \end{equation} 
        
        Now, let $j \in \ZZ_d$ and set $i = 2j-2$. Then \eqref{eq:rank_count_modx} and \eqref{eq:rank_count_mody} imply that 
            \[nt_j = \rank_\Bbbk(\epsilon_{2j-2}(j,j)\otimes_S \Bbbk) = ms_j.\]
        On the other hand, if $i = 2j-1$, then \eqref{eq:rank_count_modx} and \eqref{eq:rank_count_mody} imply that 
            \[nt_{j+1} = \rank_\Bbbk (\epsilon_{2j-1}(j+1,j+1)\otimes_S \Bbbk) = ms_j.\]
        Together, these imply that $nt_i = ms_j$ for all $i,j \in \ZZ_d$, and that $t\coloneqq t_1=t_2=\dots=t_d$ and $s \coloneqq s_1=s_2=\dots=s_d$. This implies that the matrix factorization $e(X\wotimes{} Y)$ has rank $\sum_{j=1}^d nt_j = dnt$ which is also equal to $\sum_{j=1}^d ms_j = dms$. Hence, the rank of the summand $e(X\wotimes{} Y)$ must be at least $\lcm(dn,dm) = dnm/r$. Since $\rank X\wotimes{} Y = dnm$, and any summand must have rank at least $dnm/r$, we have $\#(X\wotimes{} Y) \leq r$.
    \end{proof}

    In the following we will make repeated use of the following consequence of Nakayama's lemma: If $h: G \to F$ is homomorphism between finitely generated free $S$-modules of the same rank, then $h$ is an isomorphism if and only if $h \otimes 1_{S_1}$ is an isomorphism. The same holds for $h \otimes 1_{S_2}$, and in particular, $h, h \otimes 1_{S_1}$, and $h \otimes 1_{S_2}$ are isomorphisms simultaneously. 

    \begin{theorem}\label{thm:rel_prime}
        Let $\{u_1,u_2,\dots,u_{d-1},u_0\}$ be pairwise relatively prime elements of the maximal ideal of $S_1$ and set $X = (u_1,u_2,\dots,u_{d-1},u_0) \in \matfac{S_1}{d}{u_1u_2\cdots u_{d-1}u_0}$. If $Y \in \matfac{S_2}{d}{g}$ satisfies $Y \cong TY$, then $X \wotimes{} Y$ is indecomposable.
    \end{theorem}

    \begin{proof}
        Notice that since $u_i,u_j$ are relatively prime for $i\neq j$, if $\alpha: (T^jX)^m \to (T^iX)^m$ is a morphism, then $\alpha_k \otimes_{S_1}\Bbbk= 0$ for all $k \in \ZZ_d$. Indeed, if there exists $S_1$-homomorphisms satisfying $\beta u_i = u_j \gamma$, then neither of $\beta$ nor $\gamma$ can have unit entries.

        Set $Z = X \wotimes{} Y$ and let $\epsilon \in \End(Z)$ be an idempotent. It suffices to show that either $\epsilon = 0$ or $\epsilon = 1_Z$. 
        
        We first consider reduction modulo $(y)$ which gives us an idempotent $\epsilon_y \in \End(Z_y)\cong \End(\bigoplus_{i=1}^d(T^{1-i}X)^m)$. By the initial observation, for any $i\neq j$, the components of $\epsilon_y(i,j)$ have entries only in the maximal ideal of $S_1$. More specifically, Lemma \ref{thm:morphisms_mod_x/y}\ref{thm:morphisms_mod_y} gives us that
            \[\epsilon_y(i,j) = (\epsilon_k(2-i+k,2-j+k)_y)_{k=0}^{d-1},\]
        and therefore, $\epsilon_k(2-i+k,2-j+k)_y \otimes \Bbbk = 0$ for any $i\neq j$ and $k \in \ZZ_d$. Letting $i\neq j$ and $k$ vary over $\ZZ_d$, and lifting back to $S$, we have
            \begin{equation}\label{eq:rel_prime_1}
                \epsilon_k(i,j) \otimes \Bbbk = 0 \text{ for all } i,j,k \in \ZZ_d \text{ with } i\neq j.
            \end{equation}
            
        Next, we consider reduction modulo $(x)$. Note that since $\rank X = 1$, $\epsilon_x$ is an idempotent in $\End(Z_x) \cong \End(\bigoplus_{i=1}^d \zeta^{i-1}T^{1-i}Y)$. Furthermore, since $TY \cong Y$, we have that $Z_x$ is isomorphic to a direct sum of $d$ copies of $Y$. Let $\Lambda = \End(Y)$ which is a (likely non-commutative) local ring in the sense that $\Lambda/\rad \Lambda$ is a division ring.

        By \eqref{eq:rel_prime_1}, $\epsilon_k(i,j)_x$ is not an isomorphism for all $i,j,k \in \ZZ_d$ with $i \neq j$. It follows from Lemma \ref{thm:morphisms_mod_x/y}\ref{thm:morphisms_mod_x} that $\epsilon_x(i,j)$ is not an isomorphism of matrix factorizations for all $i \neq j$. Since $\Lambda$ is local, this implies that $\epsilon_x(i,j) \in \rad\Lambda$ for all $i \neq j$, and in particular, $\epsilon_x(i,k)\epsilon_x(k,j) \in \rad\Lambda$ for all $i,j,k$ such that $i\neq k$ or $j \neq k$. Hence Lemma \ref{thm:idempotentlemma} applies; the diagonal component $\epsilon_x(i,i)$ is an idempotent in the division ring $\Lambda/\rad\Lambda$ for each $i \in \ZZ_d$. In other words, for each $i \in \ZZ_d$, $\epsilon_x(i,i)$ is either an automorphism of $Y$ or in $\rad \Lambda$.

        First, consider the case when $\epsilon_x(i,i) \in \rad\Lambda$ for all $i \in \ZZ_d$. Since we showed that $\epsilon_x(i,j) \in \rad\Lambda$ for all $i\neq j$ above, this implies that $\epsilon_x \in \rad\End(Z_x)$. An idempotent in the radical must be zero and so we conclude that $\epsilon_x = 0$ in this case. It follows that $\epsilon \in \rad\End(Z)$ and, using the same reasoning, we have $\epsilon = 0$.
        
        Next, assume $\epsilon_x(q,q)$ is an automorphism for some $q \in \ZZ_d$. We will prove that this implies that $\epsilon_x(i,i)$ is an automorphism for all $i \in \ZZ_d$. Since $\epsilon_x(q,q)$ is an isomorphism, Lemma \ref{thm:morphisms_mod_x/y}\ref{thm:morphisms_mod_x} implies that $\epsilon_k(q,q)_x$ is an isomorphism for all $k \in \ZZ_d$ which in turn implies that $\epsilon_k(q,q)$ is an isomorphism for all $k \in \ZZ_d$.

        Let $i,j \in \ZZ_d$. To finish the proof, it suffices to show that $\epsilon_i(j,j)$ is an isomorphism. In this case, combined with \eqref{eq:rel_prime_1}, this would imply that the idempotent $\epsilon$ is an isomorphism (as it would be a unit plus a radical element) which is only possible if it is the identity morphism on $Z$. 
        
        To show the final claim, we return for a final time to reduction modulo $(y)$. Since $X$ has rank one, any endomorphism $\beta = (\beta_0,\beta_1,\dots,\beta_{d-1}) \in \End((T^{1-p} X)^m), p \in \ZZ_d$, has the property that $\beta_0 = \beta_1 = \cdots = \beta_{d-1}$. Thus, invoking Lemma \ref{thm:morphisms_mod_x/y}\ref{thm:morphisms_mod_y} again, we have 
            \[\epsilon_i(j,j)_y = \epsilon_{i+t}(2-j+t,2-j+t)_y \text{ for any }t \in \ZZ_d.\]
        In particular, for $t = q+j-2$, we have $\epsilon_i(j,j)_y = \epsilon_{i+t}(q,q)_y$, which is an isomorphism by assumption. By Nakayama's lemma, $\epsilon_i(j,j)$ is an isomorphism and this completes the proof.
    \end{proof}

    Using similar techniques one can also show the following.

    \begin{theorem}\label{thm:non_symm_otimes_rank_1}
        Let $u_1,u_2,\dots,u_{d-1},u_0 \in S_2$ and suppose $Y = (u_1,u_2,\dots,u_{d-1},u_0) \in \matfac{S_2}{d}{u_1u_2\cdots u_{d-1}u_0}$ is a rank one factorization of the product $u_1u_2\cdots u_{d-1}u_0$. If $X \in \matfac{S_1}{d}{f}$ satisfies $T^i(X) \not\cong X$ for all $i\neq 0$ in $\ZZ_d$, then $X \wotimes{} Y$ is indecomposable. \qed
    \end{theorem}

   \newpage

    \subsection{Strongly indecomposable matrix factorizations}
    
    In this section we obtain additional indecomposability results by introducing the notion of \textit{strongly indecomposable}\footnote{An object $X$ is sometimes called strongly indecomposable if its endomorphism ring is local. Our definition implies this but is stronger in general.} matrix factorizations. We continue using the assumptions and notations in \eqref{notation2}.

    \begin{definition}\label{def:strongly_indecomp}
        Let $(X,\phi) \in \matfac{S}{d}{f}$. The factorization $X$ is said to be \textit{strongly indecomposable} if it satisfies the following two conditions.
            \begin{enumerate}[label = (\roman*)]
                \item \label{def:strongly_indecomp_1} If $(\alpha,\beta)$ is a pair of $S$-homomorphisms satisfying $\alpha\phi_j = \phi_k\beta$ for some $j \neq k$, then $\alpha \otimes \Bbbk = \beta \otimes \Bbbk = 0$.
                \item \label{def:strongly_indecomp_2} If $(\alpha,\beta)$ is a pair of $S$-homomorphisms satisfying $\alpha \phi_k = \phi_k \beta$ for some $k \in \ZZ_d$, then $\alpha \otimes \Bbbk = \xi \cdot 1_{(X_{k-1}\otimes \Bbbk)}$ and $\beta \otimes \Bbbk = \xi \cdot 1_{(X_k\otimes \Bbbk)}$ for some $\xi \in \Bbbk$.
            \end{enumerate}
    \end{definition}

    In general, the assumption that a factorization is strongly indecomposable is significantly more restrictive than indecomposability. The lemma below illustrates this fact. However, in special cases of interest, strongly indecomposable matrix factorizations are not hard to find. In particular, the main goal of this section is to show that the tensor product construction preserves strong indecomposability.

    \begin{lemma}\label{thm:si_corollary}
        Suppose $(X,\phi) \in \matfac{S}{d}{f}$ is strongly indecomposable and set $R=S/(f)$. Then the following hold.
            \begin{enumerate}[label = (\roman*)]
                \item\label{thm:si_corollary_1} $X$ is an indecomposable matrix factorization.
                \item\label{thm:si_corollary_2} $T^iX \not\cong X$ for all $i \neq 0$ in $\ZZ_d$.
                \item\label{thm:si_corollary_3} $\cok\phi_i$ is an indecomposable (maximal Cohen-Macaulay) $R$-module for all $i \in \ZZ_d$.
                \item\label{thm:si_corollary_4} $\End(X)/\rad\End(X) \cong \Bbbk$ and $\End_R(\cok\phi_i)/\rad\End_R(\cok\phi_i) \cong \Bbbk$ for all $i \in \ZZ_d$.
            \end{enumerate}
    \end{lemma}

    \begin{proof}
        First notice that \ref{thm:si_corollary_2} follows directly from \ref{def:strongly_indecomp}\ref{def:strongly_indecomp_1}. Statements \ref{thm:si_corollary_3} and \ref{thm:si_corollary_1} follow from \ref{thm:si_corollary_4} so it suffices to prove the final two claims.

        To see the first isomorphism in \ref{thm:si_corollary_4}, let $\alpha \in \End(X)$. It follows from \ref{def:strongly_indecomp}\ref{def:strongly_indecomp_2} that there exists $\xi \in \Bbbk$ and $\alpha' \in \End(X)$ such that $\alpha = \xi \cdot 1_X + \alpha'$ and such that $\alpha'_k \otimes \Bbbk = 0$ for each $k \in \ZZ_d$. Since $\alpha'_k \otimes \Bbbk = 0$, the map $1 - \sigma \alpha'_k \tau$ is an isomorphism for any $S$-homomorphisms $\sigma,\tau$. In other words, the fact that $\alpha'_k \otimes \Bbbk = 0$ for all $k \in \ZZ_d$ implies that $\alpha' \in \rad\End(X)$. Thus, every element $\alpha \in \End(X)$ can be written uniquely as 
            \(\alpha = \xi\cdot 1_X + \alpha'\)
        for some $\xi \in \Bbbk$ and $\alpha' \in \rad\End(X)$. The ring map which sends $\alpha = \xi \cdot 1_X + \alpha' \mapsto \xi \in \Bbbk$ establishes the first isomorphism in \ref{thm:si_corollary_4}.

        To prove the second isomorphism, we fix $i \in \ZZ_d$ and let $g \in \End_R(M)$ where $M = \cok\phi_i$. The map $g$ lifts to a pair of $S$-homomorphisms $(\alpha,\beta)$ given by the diagram
            \[
            \begin{tikzcd}
                0 \rar &X_i \ar[d,dotted,"\beta"] \rar{\phi_i} &X_{i-1} \ar[d,dotted,"\alpha"] \rar &M \dar{g} \rar &0\\
                0 \rar &X_i \rar{\phi_i} &X_{i-1} \rar &M \rar &0.
            \end{tikzcd}
            \]
        By \ref{def:strongly_indecomp}\ref{def:strongly_indecomp_2}, there exists $\xi \in \Bbbk$ such that $\alpha = \xi \cdot 1_{X_{i-1}} + \alpha'$ and $\beta = \xi \cdot 1_{X_{i}} + \beta'$ for some $(\alpha',\beta')$ satisfying $\alpha'\otimes \Bbbk = 0$ and $\beta' \otimes \Bbbk = 0$. It follows that $g = \xi \cdot 1_M + g'$ for some $g' \in \End_R(M)$ satisfying $g' \otimes \Bbbk = 0$, that is, $\image g' \subseteq \mf n M$ where $\mf n$ denotes the maximal ideal of $S$. 
        
        Notice that if $g = \xi \cdot 1_M + g'$ with $\xi \neq 0$, then $\alpha$ and $\beta$, and consequently $g$, are each isomorphisms. In other words, the (two-sided) ideal $I = \{h \in \End_R(M) : \image h \subseteq \mf nM\}$ of $\End_R(M)$ consists exactly of the non-units of $\End_R(M)$. This implies that $\End_R(M)$ is a (possibly non-commutative) local ring and that $I = \rad\End_R(M)$. Once again, the ring homomorphism $\End_R(M) \to \Bbbk$ which sends $g = \xi \cdot 1_M + g' \mapsto \xi \in \Bbbk$ establishes the desired isomorphism.
    \end{proof}
    
    Items \ref{thm:si_corollary_3} and \ref{thm:si_corollary_4} of Lemma \ref{thm:si_corollary} are especially important consequences of strong indecomposability since, in general, the modules associated to an indecomposable $d$-fold matrix factorization (for $d > 2$) need not be indecomposable (see \cite[Example 5.5]{leuschke2023branched}). We will utilize Lemma \ref{thm:si_corollary} along with the following theorem to produce indecomposable maximal Cohen-Macaulay and Ulrich modules in Section \ref{sec:Ulrich}.

    \begin{theorem}\label{thm:strong_indecomp}
        Let $X \in \matfac{S_1}{d}{f}$ and $Y \in \matfac{S_2}{d}{g}$ be strongly indecomposable matrix factorizations. Then $X \wotimes{} Y$ is strongly indecomposable.
    \end{theorem}

    To prove Theorem \ref{thm:strong_indecomp} we need a lemma. Note that the faithful embeddings $\matfac{S_1}{d}{f} \hookrightarrow \matfac{S}{d}{f}$ and $\matfac{S_2}{d}{g} \hookrightarrow \matfac{S}{d}{g}$ preserve strong indecomposability.

    \begin{lemma}\label{thm:strong_indecomp_lemma}
        Let $(X,\phi) \in \matfac{S_1}{d}{f}$ and $(Y,\psi) \in \matfac{S_2}{d}{g}$ be strongly indecomposable of rank $n$ and $m$, respectively, and let $j,k \in \ZZ_d$.
            \begin{enumerate}[label = (\roman*)]
                \item\label{thm:strong_indecomp_lemma_1} Assume $j \neq k$. If $\alpha$ and $\beta$ are $S$-homomorphisms satisfying $\alpha(\phi_j \otimes 1_{Y_{p}}) = (\phi_k \otimes 1_{Y_{p'}})\beta$ for some $p,p' \in \ZZ_d$, then $\alpha \otimes \Bbbk = \beta \otimes \Bbbk = 0$. Similarly, if $\alpha(1_{X_p}\otimes \psi_j) = (1_{X_{p'}}\otimes \psi_k)\beta$ for some $p,p' \in \ZZ_d$, then $\alpha \otimes \Bbbk = \beta \otimes \Bbbk = 0$.
                \item\label{thm:strong_indecomp_lemma_2} If $\alpha,\beta$, and $\gamma$ are $S$-homomorphisms satisfying $\alpha(\phi_j \otimes 1_{Y_{k-1}}) = (\phi_j \otimes 1_{Y_{k-1}})\beta$ and $\alpha(1_{X_{j-1}} \otimes \psi_k) = (1_{X_{j-1}}\otimes \psi_k)\gamma$, then there exists $\xi \in \Bbbk$ such that $\alpha \otimes \Bbbk, \beta\otimes \Bbbk$, and $\gamma \otimes \Bbbk$ are each given by multiplication by $\xi$.
            \end{enumerate}
    \end{lemma}

    \begin{proof}
        By identifying $X_j \otimes Y_p$ and $X_{j-1}\otimes Y_p$ with $X_j^{\oplus m}$ and $X_{j-1}^{\oplus m}$, respectively, we may view $\phi_j\otimes 1_{Y_p}$ as an $m \times m$ block diagonal map. Similarly, we view $\phi_k \otimes 1_{Y_{p'}}: X_k^{\oplus m} \to X_k^{\oplus m}$ as an $m \times m$ block diagonal map. If we view the components of $\alpha$ and $\beta$ with respect to these direct sum decompositions, we find that $\alpha(u,v)\phi_j = \beta(u,v)\phi_k$ for each $1 \leq u,v \leq m$. Since $X$ is strongly indecomposable, the first statement of \ref{thm:strong_indecomp_lemma_1} now follows. The second statement of \ref{thm:strong_indecomp_lemma_1} uses the same argument.

        For $\ref{thm:strong_indecomp_lemma_2}$, we first identify $X_{j-1}\otimes Y_k \cong Y_k^{\oplus n}$ and $X_{j-1}\otimes Y_{k-1} \cong Y_{k-1}^{\oplus n}$. With respect to these bases, $1_{X_{j-1}} \otimes \psi_k$ is an $n\times n$ block diagonal map. If we write $\alpha$ and $\beta$ with respect to these decompositions, it follows that $\alpha(u,v)\psi_k = \psi_k \gamma(u,v)$ for all $1\leq u,v \leq n$. Since $Y$ is strongly indecomposable, there exists $\xi_{u,v} \in \Bbbk$ such that, for each pair $1\leq u,v \leq n$, $\alpha(u,v) \otimes \Bbbk$ and $\gamma(u,v) \otimes \Bbbk$ are both given by multiplication by $\xi_{u,v}$.

        Next, we consider $\phi_j \otimes 1_{Y_{k-1}}$ with respect to the same bases as in the previous paragraph, namely, $\phi_j\otimes 1_{Y_{k-1}}: Y_{k-1}^{\oplus n} \to Y_{k-1}^{\oplus n}$. Note that with these decompositions, $\phi_j\otimes 1_{Y_{k-1}}$ is not diagonal. Instead, the $(u,v)$-component is given by $\phi_j(u,v)1_{Y_{k-1}}$ for any $1 \leq u,v\leq n$. With this setup, the equation $\alpha(\phi_j \otimes 1_{Y_{k-1}}) = (\phi_j \otimes 1_{Y_{k-1}})\beta$ implies that   
            \begin{equation}\label{eq:strong_indecomp_lemma}
                \sum_{i=1}^n \alpha(u,i)\phi_j(i,v) = \sum_{i=1}^n \phi_j(u,i)\beta(i,v)
            \end{equation}
        for any $1 \leq u,v \leq n$. Let $\alpha_{s,t}$ denote the submatrix of $\alpha$ formed by collecting the $(s,t)$ entry of each block $\alpha(u,v)$, $1\leq u,v\leq n$. Similarly, we define $\beta_{s,t}$. By considering the $(s,t)$ entry of either side, equation \eqref{eq:strong_indecomp_lemma} implies that $\alpha_{s,t}\phi_j = \phi_j\beta_{s,t}$. Since $X$ is strongly indecomposable, there exists $\xi_{s,t} \in \Bbbk$ such that $\alpha_{s,t}$ and $\beta_{s,t}$ are both given by multiplication by $\xi_{s,t}$. It follows that $\alpha(u,u)\otimes \Bbbk = \alpha(v,v) \otimes \Bbbk = \beta(u,u) \otimes \Bbbk = \beta(v,v) \otimes \Bbbk$ for any $1\leq u,v \leq n$, and $\alpha(u,v) \otimes \Bbbk = \beta(u,v) \otimes \Bbbk = 0$ for all $u\neq v$. The desired conclusion of \ref{thm:strong_indecomp_lemma_2} now follows by combining this with the analysis of $\alpha$ and $\gamma$ above.
        
    \end{proof}

    \begin{proof}[Proof of Theorem \ref{thm:strong_indecomp}]
        Set $X \wotimes{} Y = (X\wotimes{} Y, \Phi)$ and let $j,k \in \ZZ_d$. Assume that $\alpha$ and $\beta$ are $S$-homomorphisms satisfying
            \begin{equation}\label{eq:strong_indecomp_1}
                \alpha\Phi_j = \Phi_k\beta.
            \end{equation}
        Write $\alpha,\beta,\Phi_j,$ and $\Phi_k$ with respect to the direct sum decompositions given in \eqref{eq:tensor}. Then the $(s,t)$-component of either side of \eqref{eq:strong_indecomp_1} gives
            \begin{align*}
                \alpha(s,t-1)(\phi_{t-1}\otimes 1_{Y_{j-t+1}}) + \alpha(s,t)(\zeta^{t-1}1_{X_{t-1}}\otimes \psi_{j-t+1}) = \\
                (\zeta^{s-1}1_{X_{s-1}}\otimes \psi_{k-s+1})\beta(s,t) + (\phi_s \otimes 1_{Y_{k-s}})\beta(s+1,t)
            \end{align*}
        Reducing this modulo $y$, we obtain the equation
            \begin{equation}\label{eq:strong_indecomp_2}
                \alpha(s,t-1)_y(\phi_{t-1}\otimes 1_{Y_{j-t+1}})_y = (\phi_s \otimes 1_{Y_{k-s}})_y\beta(s+1,t)_y,
            \end{equation}
        while reducing modulo $x$ gives
            \begin{equation}\label{eq:strong_indecomp_3}
                \alpha(s,t)_x(\zeta^{t-1}1_{X_{t-1}}\otimes \psi_{j-t+1})_x = (\zeta^{s-1}1_{X_{s-1}}\otimes \psi_{k-s+1})_x\beta(s,t)_x.
            \end{equation}
        We consider three cases in order to show that the pair $(\alpha,\beta)$ satisfies both conditions of Definition \ref{def:strongly_indecomp}.
            \begin{enumerate}
                \item \label{thm:strong_indecomp_case1} 
                    Assume that $s \neq t-1$ in $\ZZ_d$. Then, by applying Lemma \ref{thm:strong_indecomp_lemma}\ref{thm:strong_indecomp_lemma_1} to \eqref{eq:strong_indecomp_2}, we get that 
                        \[\alpha(s,t-1)_y \otimes \Bbbk = 0 = \beta(s+1,t)_y \otimes \Bbbk.\]
                    This immediately implies that $\alpha(s,t-1)\otimes \Bbbk = 0 = \beta(s+1,t)\otimes \Bbbk$. Thus, for any $s' \neq t'$, we have that $\alpha(s',t') \otimes \Bbbk = 0 = \beta(s',t') \otimes \Bbbk$.

                \item \label{thm:strong_indecomp_case2} 
                    Assume that $j \neq k$ and set $s=t$ in \eqref{eq:strong_indecomp_3}. In this case, we have
                        \[\alpha(s,s)_x(\zeta^{s-1}1_{X_{s-1}}\otimes \psi_{j-s+1})_x = (\zeta^{s-1}1_{X_{s-1}}\otimes \psi_{k-s+1})_x\beta(s,s)_x.\]
                    Again, we apply Lemma \ref{thm:strong_indecomp_lemma}\ref{thm:strong_indecomp_lemma_1}. It follows that, when $j \neq k$, $\alpha(s,s) \otimes \Bbbk = 0 = \beta(s,s) \otimes \Bbbk$ for all $s$. This, combined with \eqref{thm:strong_indecomp_case1}, establishes the first requirement of Definition \ref{def:strongly_indecomp}.

                \item \label{thm:strong_indecomp_case3}
                    Assume that $j = k$. Setting $s=t$ in \eqref{eq:strong_indecomp_3}, we obtain the equation
                        \[\alpha(s,s)_x(\zeta^{s-1}1_{X_{s-1}}\otimes \psi_{j-s+1})_x = (\zeta^{s-1}1_{X_{s-1}}\otimes \psi_{j-s+1})_x\beta(s,s)_x.\]
                    Since $j=k$, we may apply Lemma \ref{thm:strong_indecomp_lemma}\ref{thm:strong_indecomp_lemma_2} to conclude that $\alpha(s,s) \otimes \Bbbk$ and $\beta(s,s) \otimes \Bbbk$ are both given by multiplication by a common element $\xi_1 \in \Bbbk$.

                    On the other hand, if we set $s=t-1$ in \eqref{eq:strong_indecomp_2}, we get the equation
                        \[\alpha(s,s)_y(\phi_{s}\otimes 1_{Y_{j-s}})_y = (\phi_s \otimes 1_{Y_{j-s}})_y\beta(s+1,s+1)_y\]
                    which implies that $\alpha(s,s) \otimes \Bbbk$ and $\beta(s+1,s+1) \otimes \Bbbk$ are both given by multiplication by an element $\xi_2 \in \Bbbk$. Letting $s$ vary over $\ZZ_d$, we find that $\xi_1 = \xi_2$ and so the second condition of Definition \ref{def:strongly_indecomp} is satisfied. This completes the proof.
            \end{enumerate}
    \end{proof}

    \begin{example}\label{ex:si_rank_one}
        Suppose $\Bbbk$ is a field containing a primitive third root of unity, $\zeta \in \Bbbk$. Let $\{f_1,f_2,f_3\}$ be pairwise relatively prime elements of $S_1 = \Bbbk\llbracket x_1,\dots,x_a\rrbracket$ and $\{g_1,g_2,g_3\}$ pairwise relatively prime elements of $S_2 = \Bbbk\llbracket y_1,\dots,y_b\rrbracket$. Then $X = (f_1,f_2,f_3) \in \matfac{S_1}{3}{f_1f_2f_3}$ and $Y = (g_1,g_2,g_3) \matfac{S_2}{3}{g_1g_2g_3}$ are each rank one strongly indecomposable reduced matrix factorizations. By Theorem \ref{thm:strong_indecomp}, the tensor product
            \[X \wotimes{} Y = \left(\begin{pmatrix}
                g_1 & f_1 & 0\\
                0   & \zeta g_3 & f_2\\
                f_3 & 0 & \zeta^2 g_2
            \end{pmatrix},
            \begin{pmatrix}
                g_2 & f_1 & 0\\
                0   & \zeta g_1 & f_2\\
                f_3 & 0 & \zeta^2 g_3
            \end{pmatrix},
            \begin{pmatrix}
                g_3 & f_1 & 0\\
                0   & \zeta g_2 & f_2\\
                f_3 & 0 & \zeta^2 g_1
            \end{pmatrix}\right)\]
        is a strongly indecomposable reduced matrix factorization of $f_1f_2f_3 + g_1g_2g_3$ with rank three.

        More generally, suppose $f$ can be written in the form 
            \[f = \sum_{i=1}^N f_{i1}f_{i2}\cdots f_{id}\]
        where $\{f_{i1},\dots,f_{id}\}$ are pairwise relatively prime elements of a power series ring over $\Bbbk$ such that each term of the sum comes from a distinct set of variables. Then the tensor product $\bigotimes_{i=1}^N (f_{i1},\dots,f_{id})$ is a strongly indecomposable $d$-fold matrix factorization of $f$ of rank $d^{N-1}$.
    \end{example}

\section{Constructing indecomposable Ulrich modules}\label{sec:Ulrich}

    The goal of this section is to construct \textit{Ulrich modules} (and other maximal Cohen-Macaulay modules) over a hypersurface domain using the tensor product construction from Section \ref{sec:tensor_prod}. Several of the key results and methods of this section can be found either explicitly or implicitly in \cite{herzog1991linear}. For instance, they prove the existence of Ulrich modules in greater generality (over any \textit{strict complete intersection}). Our approach produces Ulrich modules over a hypersurface domain using only elementary properties of the tensor product of matrix factorizations and \cite[Proposition 5.6]{eisenbud1980homological} (see also \cite[Remark 8.8]{leuschke2012cohen}). Using the results of Section \ref{sec:disjoint_vars}, we are to prove that the Ulrich modules produced are indecomposable over hypersurface rings of a certain form. 
    
    Let $(Q,\mf n)$ be a local ring and $f \in \mf n$ a nonzerodivisor. Assume further that $Q$ is a $\ZZ[\zeta]$-algebra where $\zeta \in \CC$ is a primitive $d^{\textit{th}}$ root of unity. Recall that a matrix factorization $(X,\phi) \in \matfac{Q}{d}{f}$ is reduced if $\image \phi_k \subseteq \mf n X_{k-1}$ for each $k \in \ZZ_d$. This means that, after choosing bases, the entries of $\phi_k$ lie in the maximal ideal of $Q$ for each $k \in \ZZ_d$.

    \begin{proposition}\label{thm:det_formula_modules}
        Suppose that $I \subset Q$ is an ideal such that $f \in I^d$ for some $d \ge 2$. If $f$ can be written as $f = \sum_{i=1}^N f_{i1}f_{i2}\cdots f_{id}$ for $f_{ij} \in I$ and some $N \ge 2$, then there exists a reduced matrix factorization $(X,\phi) \in \matfac{Q}{d}{f}$ such that
        \begin{enumerate}[label = (\roman*)]
            \item $\rank X = d^{N-1}$, and
            \item $\det \phi_k = \pm f^{d^{N-2}}$ for any $k \in \ZZ_d$.
        \end{enumerate}
    \end{proposition}

    \begin{proof}
        For each integer $1 \leq i \leq N$, let $Y_i = (f_{i1},f_{i2},\dots,f_{id})$ be the rank one matrix factorization of the product $f_i \coloneqq f_{i1}f_{i2}\cdots f_{id}$ with $d$ factors. In other words, $Y_i \in \matfac{Q}{d}{f_i}$ for each $1 \leq i \leq N$. Then the tensor product $X = \bigotimes_{i=1}^N Y_i$ is a reduced matrix factorization of $f = \sum_{i=1}^N f_i$ of rank $d^{N-1}$. The determinant formula follows from Proposition \ref{thm:det_formula}.
    \end{proof}

    A finitely generated module $M$ over a local ring $(A,\mf m)$ is called \textit{maximal Cohen-Macaulay} (MCM) if $\depth_A M = \dim A$. We will let $\mu_A(M)$ denote the number of elements in a minimal generating set of $M$ and the Hilbert-Samuel multiplicity of an $A$-module will be denoted $e_A(M)$. If $f \in \mf m$, $\ord(f)$ will denote the order of $f$, that is, largest integer $d$ such that $f \in m^d$. Recall that if $A$ is a regular local ring and $f \in \mf m$ has order $d$, then the multiplicity of the ring $A/(f)$ is exactly $d=\ord(f)$ (see \cite[Example 11.2.8]{huneke2006integral}).
    
    If $(Q,\mf n)$ is a Cohen-Macaulay local ring ($\depth_Q Q = \dim Q)$ and $f \in \mf n$ is a nonzerodivisor, then every module associated to a matrix factorization $(X,\phi) \in \matfac{Q}{d}{f}$ defines an MCM $R = Q/(f)$-module. More specifically, $\cok (\phi_k\phi_{k+1}\cdots\phi_{k+r-1})$ is an MCM $R$-module for each $k \in \ZZ_d$ and $1 \leq r \leq d$; this follows directly from the Auslander–Buchsbaum formula.

    \begin{proposition}\label{thm:MCMs_from_tensoring}
        Assume that $(Q,\mf n)$ is a regular local ring, $f \in \mf n^2$ is irreducible, and set $R = Q/(f)$. Let $d = \ord(f)$ and write $f = \sum_{i=1}^N f_{i1}f_{i2}\cdots f_{id}$ for $f_{ij} \in \mf n, N \ge 2$. Then, for each pair of integers $2 \leq k \leq d, 1 \leq \ell \leq d-1$, there exists a maximal Cohen-Macaulay $R$-module $M$ (with no free summands) satisfying
            \begin{enumerate}[label = (\roman*)]
                \item $\mu_R(M) = k^{N-1}$
                \item $\rank_R(M) = \ell k^{N-2}$.
            \end{enumerate}
        In particular, the multiplicity of $M$ is $e_R(M) = d \ell k^{N-2}$.
    \end{proposition}

    \begin{proof}
        Fix $2 \leq k \leq d$ and $1 \leq \ell \leq d-1$. For each integer $1\leq i \leq N$, choose a partition of the product $f_i = f_{i1}f_{i2}\cdots f_{id}$ into $k$ products of the $f_{ij}$. Note that the index $i$ is not taken modulo $d$. The choice of partition can be viewed as a rank one reduced matrix factorization of $f_i$ with $k$ factors, call it $Y_i \in \matfac{Q}{k}{f_i}$. Let $(X,\phi) \in \matfac{Q}{k}{f}$ be the tensor product of $Y_1,Y_2,\dots, Y_N$ as in the proof of Proposition \ref{thm:det_formula_modules} and set $M = \cok\phi_1\phi_2\cdots\phi_{\ell}$ (the choice to end at $\phi_1$ is not important; any composition of $\ell$ maps would suffice). Then $X$ has rank $k^{N-1}$ by Proposition \ref{thm:det_formula_modules} and we have a minimal free resolution of $M$ over $Q$
            \[0 \to X_{\ell} \xrightarrow{\phi_1\phi_2\cdots \phi_\ell} X_0 \to M \to 0.\] 
        Hence, we may tensor this resolution with $R$ to obtain a minimal presentation of $M$ over $R$. It follows that $M$ is an MCM $R$-module minimally generated by $k^{N-1}$ elements. 

        Again by Proposition \ref{thm:det_formula_modules}, we have $\det\phi_j = \pm f^{k^{N-2}}$, for each $j \in \ZZ_d$. Hence $\det(\phi_1\phi_2\cdots \phi_\ell) = \pm f^{\ell k^{N-2}}$. Since $f$ is irreducible, \cite[Proposition 5.6]{eisenbud1980homological} implies that $\rank_R(M) = \ell k^{N-2}$. Finally, since $R$ is a domain, the multiplicity of $M$ may be computed as $e_R(M) = e_R(R)\rank_R(M) = d \cdot  \ell k^{N-2}$.
    \end{proof}

    The minimal number of generators of an MCM module $M$ over a local ring $A$ is always    bounded above by the modules multiplicity. In other words, $\mu_A(M) \leq e_A(M)$. Maximal Cohen-Macaulay modules for which equality is attained are called \textit{Ulrich} modules. In \cite{ulrich1984gorenstein}, Ulrich asked if every Cohen-Macaulay local ring has an Ulrich module. After remaining open for nearly fourty years, Ulrich's question was answered in the negative in \cite{iyengar2024non}. They give examples of Gorenstein, and even complete intersection, local domains which have no Ulrich modules. Earlier examples of non-Cohen-Macaulay local rings which have no Ulrich modules were give by Yhee in \cite{yhee2023ulrich}. In light of these examples, the case of a hypersurface ring stands out as unique.
    
    \begin{corollary}\cite[Theorem 2.5]{herzog1991linear}\label{thm:Ulrich_existence}
        Assume that $(Q,\mf n)$ is a regular local ring, $f \in \mf n^2$ is irreducible, and set $R = Q/(f)$. Then there exists an Ulrich $R$-module.
    \end{corollary}

    \begin{proof}
        This follows from Proposition \ref{thm:MCMs_from_tensoring} by picking $k = \ord(f)$ and $\ell = 1$.
    \end{proof}

    A subcategory $\mc B$ of an abelian category $\mc A$ is called \textit{extension closed} if it is closed under direct summands and if for each short exact sequence $0 \to L \to M \to N \to 0$, if $L,N \in \mc B$, then $M \in \mc B$ as well. As it turns out, the full subcategory of Ulrich modules over a Cohen-Macaulay local ring is rarely extension closed; in dimension one this happens if and only if the ring is regular (see \cite[Corollary 5.2.7]{dao2024exact}). The constructions of this section induce a similar result for hypersurface domains.

    \begin{corollary}
        Assume $(Q,\mf n)$ is a regular local ring, $f \in \mf n$ is irreducible, and set $R = Q/(f)$. Then the subcategory of Ulrich $R$-modules is extension closed if and only if $R$ is regular.
    \end{corollary}

    \begin{proof}
        If $R$ is regular, then every MCM module is free. Hence every Ulrich module is free and the subcategory of Ulrich modules is therefore extension closed.

        Conversely, suppose that $R = Q/(f)$ is not regular, that is, suppose that $f \in \mf n^2$. Set $d = \ord(f)$. Since $d \ge 2$, we may apply Proposition \ref{thm:MCMs_from_tensoring} to obtain a reduced matrix factorization $X= (\phi_1,\phi_2,\dots,\phi_{d-1},\phi_0) \in \matfac{Q}{d}{f}$ satisfying
            \begin{enumerate}[label = (\roman*)]
                \item $N = \cok \phi_1$ and $L = \cok\phi_{2}$ are Ulrich $R$-modules, and
                \item $M = \cok(\phi_1\phi_{2})$ satisfies $\ds\frac{\mu_R(M)}{e_R(M)} = \frac{1}{2}$. 
            \end{enumerate}
        In particular, $M$ is not Ulrich. The result now follows since the commutative diagram with exact rows
            \[
                \begin{tikzcd}
                    0 \rar &X_{2} \dar[equals] \rar{\phi_2} &X_1 \dar{\phi_1} \rar & L \dar \rar &0\\
                    0 \rar &X_2 \rar{\phi_1\phi_2} &X_0 \rar &M \rar &0
                \end{tikzcd}
            \]
        induces a short exact sequence of $R$-modules $0 \to L \to M \to N \to 0$. Once again, the choice of $\phi_1$ and $\phi_2$ was not important; any pair of consecutive maps in $X$ induce a short exact sequence of this form.
    \end{proof}

    Combining the constructions of this section and the indecomposability results of Section \ref{sec:disjoint_vars}, we obtain the following.

    \begin{theorem}\label{thm:indecomp_ulrich_mods}
        Let $Q = \Bbbk\llbracket x_1,\dots, x_n\rrbracket$ for a field $\Bbbk$, $f \in Q$ be irreducible of order $e \ge 2$, and set $R = Q/(f)$. Assume further that $\Bbbk$ contains a primitive $d^{\textit{th}}$ root of unity for all $2\leq d\leq e$. Suppose that $f$ can be written as $f = \sum_{i=1}^N f_{i1}f_{i2}\cdots f_{id}$, for $f_{ij} \in \mf (x_1,\dots,x_n)$ and $2 \leq d \leq e$, such that
            \begin{enumerate}[label = (\roman*)]
                \item for each $1\leq i,j \leq N$ with $i \neq j$, the power series $\{f_{i1},f_{i2},\dots,f_{id}\}$ and $\{f_{j1},f_{j2},\dots,f_{jd}\}$ are written in distinct sets of variables, and
                \item the rank one factorization $(f_{i1},f_{i2},\dots,f_{id}) \in \matfac{Q}{d}{f_{i1}f_{i2}\cdots f_{id}}$ is strongly indecomposable for each $1\leq i \leq N$.
            \end{enumerate}
        Then there exists an indecomposable MCM $R$-module $M$ satisfying $\rank_R(M) = d^{N-2}$, $\mu_R(M) = d^{N-1},$ and $e_R(M) = ed^{N-2}$. Moreover, if $d=e$, then $M$ is an indecomposable Ulrich $R$-module.
    \end{theorem}

    \begin{proof}
        The assumptions allow us to apply Theorem \ref{thm:strong_indecomp} to conclude that the tensor product
            \[X = \bigotimes_{i=1}^N (f_{i1},f_{i2},\dots,f_{id})\]
        is a strongly indecomposable matrix factorization of $f$ in $\matfac{Q}{d}{f}$. If $X = (\Phi_1,\Phi_2,\dots,\Phi_{d-1},\Phi_{0})$, then Lemma \ref{thm:si_corollary} implies that $\cok\Phi_i$ is and indecomposable MCM $R$-module for each $i \in \ZZ_d$. The rank, minimal number of generators, and multiplicity counts all follow as in the proof of Proposition \ref{thm:MCMs_from_tensoring}.
    \end{proof}

    Theorem \ref{thm:indecomp_ulrich_mods} can be applied to any polynomial (or power series) of the form given in Example \ref{ex:si_rank_one}. We give one specific instance of this below.

    \begin{example}\label{ex:generic_sum}
        Let $Q = \Bbbk \llbracket x_{ij} : 1\leq i \leq N, 1 \leq j \leq d\rrbracket$ where $N,d\ge 2$, and $\Bbbk$ is a field containing a primitive $d^{\textit{th}}$ root of unity. Consider the polynomial
            \[f = \sum_{i=1}^N x_{i1}^{a_{i1}}x_{i2}^{a_{i2}}\cdots x_{id}^{a_{id}} \in Q,\]
        for integers $a_{ij} \ge 1$, and assume that $f$ is irreducible. Each of the rank one matrix factorizations $X_i = (x_{i1}^{a_{i1}},x_{i2}^{a_{i2}},\dots, x_{id}^{a_{id}}) \in \matfac{Q}{d}{x_{i1}^{a_{i1}}x_{i2}^{a_{i2}}\cdots x_{id}^{a_{id}}}$, $1\leq i \leq N$, is strongly indecomposable (see Example \ref{ex:si_rank_one}). Hence, Theorem \ref{thm:strong_indecomp} implies that $X = \bigotimes_{i=1}^N X_i \in \matfac{Q}{d}{f}$ is strongly indecomposable of rank $d^{N-1}$. Set $X = (\Phi_1,\Phi_2,\dots,\Phi_0)$ and $M_i = \cok\Phi_i$, $i \in \ZZ_d$.

        If $a_{ij} = a \ge 1$ for all $i,j$, then $\ord(f) = da$. In this case, each module $M_i$, $i \in \ZZ_d$, is an indecomposable MCM $R$-module satisfying
            \[\frac{\mu_R(M_i)}{e_R(M_i)} = \frac{d^{N-1}}{(da)d^{N-2}} = \frac{1}{a}.\]
        If $a = 1$, then each $M_i$, $i \in \ZZ_d$, is an indecomposable Ulrich module minimally generated by $d^{N-1}$ elements and of rank $d^{N-2}$.
    \end{example}

    The notion of \textit{Ulrich complexity} was introduced in \cite{blaser2017ulrich} to quantify the possible ranks of Ulrich modules over hypersurface rings. Indeed, the Ulrich complexity of homogeneous polynomial of degree $d$ is the smallest integer $r$ such that there exists a matrix factorization $(A,B) \in \matfac{}{2}{f}$ satisfying $\det(A) = f^r$ and such that the entries of $A$ are linear forms. In this case, we will write $\ulcomp(f) = r$. The Ulrich complexity of several examples and classes of polynomials was computed in \cite{blaser2017ulrich} but, in general, Ulrich complexity is difficult to determine. For an example of their results, consider a generic trinomial $f = x_{11}x_{12}x_{13} + x_{21}x_{22}x_{23} + x_{31}x_{32}x_{33}$ (i.e. $N = 3, d=3$, and $a=1$ from above). They show that $\ulcomp(f) = 3$ \cite[Proposition 4.10]{blaser2017ulrich}. The matrix factorization of $f$ constructed in Example \ref{ex:generic_sum} produces Ulrich modules which attain this minimum and are (necessarily) indecomposable.

    If $f$ is a homogeneous polynomial which can be written as a sum of $N$ products of linear forms, then the tensor product construction gives an upper bound for Ulrich complexity; $\ulcomp(f) \leq d^{N-2}$ \cite[Theorem 2.1]{blaser2017ulrich}. Interestingly, Example \ref{ex:generic_sum} produces indecomposable Ulrich modules attaining this upper bound for any generic polynomial of the form $f = \sum_{i=1}^N x_{i1}x_{{i2}}\cdots x_{id}$. If the Ulrich complexity is strictly less than $d^{N-2}$ in this case, then it seems one would need a construction more efficient than the tensor product to attain the minimum.

\begin{ack}
    The first author's work on this project was supported by the University of Utah's Undergraduate Research Opportunities Program (UROP) and the Mathematics Department REU. The first author was also supported by the NSF RTG Grant \#1840190. We would also like to thank the anonymous referee for their helpful suggestions, which have improved the paper. 
\end{ack}   

\bibliographystyle{amsalpha}
\bibliography{references}

\end{document}